\documentclass[11pt,a4paper]{article}
\usepackage{pgfplots}
\usepackage{tikz}
\usetikzlibrary{datavisualization}
\usetikzlibrary{graphs}
\usepackage{epsfig}
\usepackage{graphicx}
\graphicspath{{figures/}}
\usepackage[margin=2cm]{geometry}
\usetikzlibrary{decorations.markings}
\usepackage{mathtools,amssymb,amsthm}
\usepackage{mathrsfs}
\usepackage[font=small]{caption}
\usepackage{xcolor}
\usepackage[colorlinks=true,citecolor=red]{hyperref}
\usepackage{subfigure}
\usepackage{cite}
\usepackage[capitalise,noabbrev]{cleveref}
\crefname{equation}{}{}
\numberwithin{equation}{section}
\numberwithin{figure}{section}

\usepackage[textwidth=1.8cm]{todonotes}

\newtheorem{theorem}{Theorem}[section]
\newtheorem{definition}[theorem]{Definition}
\newtheorem{lemma}[theorem]{Lemma}

\newtheorem{corollary}[theorem]{Corollary}

\newtheorem{remark}[theorem]{Remark}

\def\rme{{\mathrm e}}

\def\dif{{\mathrm d}}
\def\pdif{\partial}

\def\tm{{\widetilde m}}

\def\bpsi{{\bar\psi}}
\def\eps{{\varepsilon}}
\def\RR{{\mathbb R}}
\def\CC{{\mathbb C}}
\def\calE{{\mathcal E}}
\def\calF{{\mathcal F}}
\def\calH{{\mathcal H}}
\def\calQ{{\mathcal Q}}
\def\calJ{{\mathcal J}}
\def\Jm{\calJ_m}

\def\calL{{\mathcal L}}
\def\calO{{\mathcal O}}
\def\calT{{\mathcal T}}
\def\calA{{\mathcal A}}
\def\calM{{\mathcal M}}
\def\calR{{\mathcal R}}
\def\ess{{\rm ess}}

\def\LG{{\Lambda}}
\def\id{{\rm id}}
\def\hom{{\rm h}}

\def\Lspace{{L}}
\def\Hspace{{H}}
\def\Xspace{{X}}
\def\Yspace{{Y}}
\def\Cspace{{C}}
\def\Q{{Q}}
\def\bQ{\bar\Q}
\def\cQ{\check\Q}
\def\calX{{\mathcal X}}
\def\calY{{\mathcal Y}}
\def\Xk{\calX_\kappa}
\def\Yk{\calY_\kappa}
\def\R{\mathbb{R}}
\def\Gammah{{\Gamma_h}}
\def\vphi{\varphi}

\begin{document}

\title{The stability of smooth solitary waves for the $b$-family of Camassa-Holm equations}

%
%
%
%
%
%
%
%

\author{
Ji Li
\thanks{School of Mathematics and Statistics, Huazhong University of Science and Technology, Wuhan 430074, China ({\tt liji@hust.edu.cn}).}
\and
Changjian Liu
\thanks{School of Mathematics (Zhuhai), Sun Yat-sen University, Zhuhai 519082, China ({\tt liuchangj@mail.sysu.edu.cn}).}
\and
Teng Long
\thanks { (Corresponding author) School of Mathematics and Statistics, Huazhong University of Science and Technology, Wuhan 430074, China ({\tt longteng@hust.edu.cn}).}
\and
Jichen Yang
\thanks{School of Mathematical Sciences, Harbin Engineering University, Harbin 150001, China ({\tt jichen.yang@hrbeu.edu.cn}).}
}

\date{\today}

%

\maketitle

\begin{abstract}
The $b$-family of Camassa-Holm ($b$-CH) equation is a one-parameter family of PDEs, which includes the completely integrable Camassa-Holm and Degasperis-Procesi equations but possesses different Hamiltonian structures. Motivated by this, we study the existence and the orbital stability of the smooth solitary wave solutions with nonzero constant background to the $b$-CH equation for the special case $b=1$, whose the Hamiltonian structure is different from that of $b\neq1$. We establish a connection between the stability criterion for the solitary waves and the monotonicity of a singular integral along the corresponding homoclinic orbits of the spatial ODEs. We verify the latter analytically using the framework for the monotonicity of period function of planar Hamiltonian systems, which shows that the smooth solitary waves are orbitally stable. In addition, we find that the existence and orbital stability results for $0<b<1$ are similar to that of $b>1$, particularly the stability criteria are the same. Finally, combining with the results for the case $b>1$, we conclude that the solitary waves to the $b$-CH equation is structurally stable under the variation of 
the parameter $b>0$.
\end{abstract}

\section{Introduction}\label{s:intro}
The $b$-family of Camassa-Holm ($b$-CH) equations
\begin{equation}\label{e:bCH}
u_t-u_{txx}+(b+1)uu_x=bu_xu_{xx}+uu_{xxx},
\end{equation}
where $u=u(t,x)\in\R$ is the velocity variable and $b$ is a real parameter, was introduced in \cite{DHH2002,DGH2001} by using transformations of the integrable hierarchy of KdV equations.
The $b$-CH equation includes two special cases: $b=2$ called Camassa-Holm (CH) equation and $b=3$ called Degasperis-Procesi (DP) equation. They are unidirectional models that arise from the asymptotic approximation of the incompressible Euler equations in the case of shallow water~\cite{Iva2007}, and both of them are completely integrable with bi-Hamiltonian structures. 
The existence and orbital stability of the smooth solitary waves and peakons for these two models have been well developed, and we refer the readers to some relevant literature, e.g.,\cite{CM2001,CS2000,CS2002} for the CH equation and \cite{Lin,LLW2020,LLW2023} for the DP equation. Due to the complete integrability and special Hamiltonian structures of the cases $b=2,3$, these results cannot be extended to the cases of other values of $b$. In \cite{LP2022}, the authors study the orbital stability of smooth solitary waves to \cref{e:bCH} for any $b>1$. They provide a sufficient condition of the orbital stability, and verify analytically for $b=2,3$ and numerically for other values of $b>1$ that the stability criterion is satisfied. In the most recent work \cite{LL2023}, the authors verify analytically for any values of $b>1$ that the aforementioned stability criterion can be satisfied.

\medskip
The main purpose of this paper is to study the orbital stability of the smooth solitary waves of \cref{e:bCH} for $0<b\leq 1$. The stability analysis is essentially based on a general framework that motivated by the approach of \cite{GSS1987}, which characterizes the critical points of Hamiltonian systems with symmetries and conserved quantities. Making use of this, we obtain a stability criterion for the orbitally stable solitary waves. 
We next establish a connection between the stability criterion and the monotonicity of an integral along the corresponding homoclinic orbit. This method has been used in \cite{LP2022} for the case of $b>1$, and the stability criterion has been proved to hold in \cite{LL2023}, and we extend it to the case of $0<b\leq 1$ in this paper; in particular, we provide a complete analysis for the case $b=1$ and a series of remarks and discussions for the case $0<b<1$.  However, it is quite technical to verify the stability criterion, i.e., determining the monotonicity.

\medskip
We next briefly introduce our method which is based upon the framework of period functions in planar systems which possess first integral. For instance, we consider the Hamiltonian system
\[
\frac{\dif x}{\dif t}=-\frac{\partial H(x,y)}{\partial y}, \quad \frac{\dif y}{\dif t}=\frac{\partial H(x,y)}{\partial x},
\]
which possesses a center, i.e., any orbit surrounding the center is closed in the phase space. The \emph{period function} is defined as the minimum period of each periodic orbit, and each orbit can be parameterized by $\gamma_h:=\{(x,y): H(x,y)=h,\, h\in\Omega \}$, where $h$ represents the energy level along the orbit. The period function, denoted by $T(h)$ and for $h\in\Omega$, is expressed by definition as
\[
T(h) :=\oint_{\gamma_h}\dif t=-\oint_{\gamma_h}\frac{1}{\partial_yH}\,\dif x.
\]
It is clear that there must exist zeros of $\partial_yH$ along the closed orbit $\gamma_h$, since $\dif x/\dif t=0$ at the left-/rightmost point of the periodic orbit in the $(x,y)$-plane.
Hence, the period function $T(h)$ is a singular integral and it is difficult to derive the expression of $T'(h)$, and thus determining the monotonicity of the period function is a challenge.
Fortunately, for some special systems, $T'(h)$ can be expressed as integral form using different methods depending on the form of the systems, and thus if the integrand has the fixed sign, then $T(h)$ will be monotone, cf., e.g., \cite{Coppel,Garijo,Gasull,LLLW2022,LLW2022,Jordi,Zhao}.
In some context of the nonlinear dynamics, e.g., \cite{DLL2021,Liu2020,LL2023},  the analysis of the orbital stability of the solitary wave can be converted to the determination of the monotonicity of an integral, denoted by e.g., $\Q(h)$, along the corresponding homoclinic orbit with the energy level $h$. In contrast to the fact that $T(h)$ is singular for all $h\in\Omega$, the integral $\Q(h)$ may be singular for $h$ approaches only the boundary of $\Omega$. Nevertheless, for both situations, a way of determining the monotonicity is to ``smoothen'' the singular integral first by making the product with $h$, then differentiating $h\Q(h)$ with respect to $h$ yields the expression of $\Q'(h)$ which has the integral form and its integrand has a definite sign.
In addition, we refer the readers to~\cite{Ehrman,Geyer1,Geyer2} for the stability of smooth periodic traveling waves of \cref{e:bCH}, which also make use of the framework of period functions in Hamiltonian systems.

\medskip

Our major goal in this paper is to understand the existence and orbital stability of the solitary waves to \cref{e:bCH} for the special case $b=1$, the Hamiltonian structure of which is different from that of the case $b\neq1$. Analyzing the steady states in moving frame, we prove that the traveling wave solutions to \cref{e:bCH} have the form of solitary waves, which correspond to the homoclinic orbits of the spatial ODEs. We find that the maximum of the solitary wave over the space is increasing in the wave speed $c$, and is decreasing in the background $\kappa$. The following theorem is the main result of this paper.

\begin{theorem}[Orbital stability for $b=1$]\label{t:main}
For $b=1$ and $\kappa\in(0, c/2)$, the smooth solitary waves of \cref{e:bCH} with the nonzero constant background $\kappa$ and the traveling wave speed $c$ are orbitally stable.
\end{theorem}

The proof of this theorem is straightforward from the combination of \cref{t:stable,l:equiv,l:criterion} presented below, 
which follows from the aforementioned frameworks.

\begin{remark}[Case $0<b<1$]
In the case $0<b<1$, we find that the analyses of the existence and orbital stability of the smooth solitary waves to \cref{e:bCH} are analogous to that of the case $b>1$ presented in \cite{LP2022}; in particular, the stability criterion is the same as that of $b>1$ given by \cite[Eq.~(16)]{LP2022}. Hence, combining with the proof of the stability criterion as in \cite{LL2023}, we conclude that the smooth solitary wave with the nonzero constant background $\kappa$ and the traveling speed $c$ is orbitally stable for $0<b<1$ and $0<\kappa<\frac{c}{b+1}$, where the background $\kappa$ means $u(\cdot,x)\to\kappa$ as $x\to\pm\infty$.  Further discussions on the existence and stability results for $0<b<1$ can be found in \cref{s:existbn1,r:stability}.
\end{remark}

Combining with the results of $b>1$ in \cite{LP2022}, 
we conclude that the solitary waves are structurally stable for $b>0$, i.e., the solitary wave exists and the phase portrait of the corresponding Hamiltonian system \cref{e:ODE1} does not change qualitatively under the variation of the parameter $b>0$; we also illustrate this in \cref{f:homovaryb}. 

\medskip
This paper is organized as follows. In \cref{s:Ham} we revisit the Hamiltonian formulation of \cref{e:bCH} for both cases $b\neq1$ and $b=1$. In \cref{s:existence} we prove the existence of smooth solitary wave solutions of \cref{e:bCH} for $b=1$ and give a discussion for the case of $0<b<1$. We next study the orbital stability of the solitary waves for $b=1$ in \cref{s:stability}, particularly obtaining a stability criterion. The verification is postponed to \cref{s:criterion}. Finally, we provide a short discussion in \cref{s:discussion}. The appendices contain some technical
aspects.

\section{Hamiltonian formulation}\label{s:Ham}

We revisit the Hamiltonian formulation of \cref{e:bCH} for both $b\neq 1$ and $b=1$, and refer the readers to \cite{DHH2003,LP2022} for more details. In case of $b=1$, however, the Hamiltonian of \cref{e:bCH} on nonzero background is different from that of zero background presented in \cite{DHH2003}, thus we particularly discuss this case below. It is conventional to introduce the momentum density
\begin{align}\label{e:momentum}
m := (1-\partial_x^2)u,
\end{align}
and $u$ can be expressed as
\begin{align*}
u = (1-\partial_x^2)^{-1}m.
\end{align*}
The operator $(1-\partial_x^2)^{-1}$ is defined through the convolution
\begin{align}\label{e:convolution}
(1-\partial_x^2)^{-1} f := g * f,\quad f\in\Lspace^2(\RR),
\end{align}
where the kernel function $g=g(x)=\rme^{-|x|}/2$ is the Green's function for the Helmholtz operator $1-\partial_x^2$ on the real line, i.e., $(1-\partial_x^2)g(x) = \delta(x)$ with Dirac delta distribution $\delta(x)$.
Making use of \cref{e:momentum}, the $b$-CH equation \cref{e:bCH} can be written in the form
\begin{equation}\label{e:bm}
m_t+um_x+bmu_x=0.
\end{equation}
As mentioned in \cref{s:intro}, we study the solitary wave solutions with nonzero constant background $\kappa$, i.e., $u(\cdot,x)\to\kappa$ as $x\to\pm\infty$. Hence, we consider the function spaces
\begin{subequations}\label{e:space}
\begin{align}
\Xk\ &:=\ \{m-\kappa\in H^1(\mathbb{R}): m(x)>0,\, x\in\mathbb{R}\},\label{e:mspace}\\
\Yk\ &:=\ \{u-\kappa\in \Hspace^3(\RR): \ (1-\pdif_x^2)u(x)>0, \ x\in\RR\}.\label{e:uspace}
\end{align}
\end{subequations}

\paragraph{Case $b\neq1$.}
We emphasize that the Hamiltonian structure and conserved quantities for the case $0<b<1$ are exactly the same as that for the case $b>1$ presented in \cite{LP2022}. In this case, for $m(t,\cdot)\in\Xk$ (i.e., $u(t,\cdot)\in\Yk$) \cref{e:bm} can be cast in the Hamiltonian form
\begin{equation*}
\frac{\pdif m}{\pdif t}=\calJ_{m,b}\frac{\delta \calE}{\delta m},
\end{equation*}
where the skew-symmetric operator $\calJ_{m,b}:=-(bm\partial_x+m_x)(1-\partial^2_x)^{-1}\partial^{-1}_x(bm\partial_x +(b-1)m_x)$. The conserved (i.e., time-invariant) mass integral $\calE(m):=\frac{1}{b-1}\int_{\mathbb{R}}(m-\kappa)\,\dif x$. There are two other conserved quantities for $m(t,\cdot)\in\calX_\kappa$ given by
\[
\calF_1(m)=\frac{1}{b-1}\int_{\R}\left(m^{\frac 1 b}-\kappa^{\frac 1 b}\right)\dif x,\quad \calF_2(m)=\frac{1}{b-1}\int_{\R}\left(\left(\frac{m_x^2}{b^2m^2}+1\right)m^{-\frac{1}{b}}-\kappa^{-\frac{1}{b}}\right)\dif x.
\]

\paragraph{Case $b=1$.}
In this case, the $b$-CH equation \cref{e:bCH} is of the form
\begin{equation}\label{e:CH}
u_t-u_{txx}+2uu_x=u_xu_{xx}+uu_{xxx}.
\end{equation}
Again, making use of \cref{e:momentum},
\cref{e:CH} can be expressed as
\begin{equation}\label{e:m}
m_t+ (um)_x=0.
\end{equation}
For $m(t,\cdot)\in\Xk$, \cref{e:m} can be written in the Hamiltonian form
\begin{equation}\label{e:CHham}
\frac{\partial m}{\partial t}=\Jm\frac{\delta\calH}{\delta m}(m),
\end{equation}
where the skew-symmetric operator $\Jm:=\calJ_{m,1}$ and takes the form
\begin{align}\label{e:J}
\Jm= -\partial_x\big(m(1-\partial_x^2)^{-1}\partial_x^{-1}(m\partial_x\cdot)\big),
\end{align}
and the conserved integral
\begin{equation}\label{e:H}
\calH(m)=\int_{\mathbb{R}}\Big(m(\ln m-\ln \kappa)-(m-\kappa)\Big)\dif x.
\end{equation}
In comparison with the Hamiltonian $\int_\R m\ln m\,\dif x$ for the case of zero background \cite{DHH2003}, the configuration of \cref{e:H} is more subtle. We note that this is a linear combination of the conserved integrals $\int_\R m\ln m\,\dif x$ and $\int_\R m\,\dif x$.
Indeed, \cref{e:m} implies that $\int_\RR m\,\dif x$ is conserved, and making use of \cref{e:m}, we have
\[
(m\ln m)_t=m_t(1+\ln m)=-(um)_x(1+\ln m)=-\partial_x\left(um\ln m+\frac{1}{2}u^2-\frac{1}{2}u_x^2\right).
\]
Alternatively, one may obtain the conserved quantity $\int_\R m\ln m\,\dif x$ from the limit of the difference $\calE(m)-\calF_1(m)$ for $\kappa=0$ as $b\to1$, i.e., $\lim_{b\to1}[\frac{1}{b-1}(m-m^{1/b})]= m\ln m$.
There are two other conserved quantities given by
\begin{align}\label{e:Qj}
\calQ_1(m) := \int_\RR (m - \kappa) \dif x,\quad \calQ_2(m) := \int_\RR \Big(m^{-3}(m^2+m_x^2) - \kappa^{-1}\Big)\dif x.
\end{align}
These arise from the translation symmetry of the traveling waves; moreover, the variational derivatives of these conserved quantities are detailed in \cref{s:Casimirs}.

\section{Existence of solitary waves}\label{s:existence}

In this section, we consider the existence of smooth solitary wave solutions to \cref{e:bCH} on nonzero background $\kappa$ for $b=1$ by analyzing its spatial dynamics. In addition, we also present the result of the case $0<b<1$.

\subsection{Case $b=1$}
We seek the solitary wave solutions of the form
\begin{align}\label{e:sw}
u(x,t)=\phi(\xi)
\end{align}
for wave shape $\phi$ and phase variable $\xi:=x-ct$, where $c$ is the wave speed and it is assumed to be positive without loss. Accordingly, the momentum density $m$ in the moving frame takes the form
\begin{align}\label{e:tw}
m(x,t) = \mu(\xi) := \phi(\xi) - \phi_{\xi\xi}(\xi).
\end{align}
We assume that the solitary wave $\phi(\xi)$ approaches the asymptote $\kappa$ as $\xi\to\pm\infty$, where $\kappa$ is assumed to be positive without loss, and thus the associated momentum density has the limit $\mu(\xi)\to \kappa$ as $\xi\to\pm\infty$. We remark that \cref{e:CH} does not possess smooth solitary waves with zero asymptote $\kappa$, see more discussion in \cref{r:ratio} below.

With the profiles $\phi$ and $\mu$ in the moving frame $\xi$, \cref{e:m} can be transformed into the traveling wave ODE
\begin{align}\label{e:ODE}
-c\mu_\xi + (\phi\mu)_\xi = 0.
\end{align}
Integrating \cref{e:ODE} with respect to $\xi$, yields
\begin{align}\label{e:integratedeqn}
(c-\phi)\mu = (c-\kappa)\kappa,
\end{align}
where the integration constant on the right-hand side stems from the aforementioned asymptotes of $\phi$ and $\mu$.
With the relation $\mu = \phi - \phi_{\xi\xi}$ and for $\phi\neq c$, we can write \cref{e:integratedeqn} as a second-order ODE
\begin{align}\label{e:ODE2}
\phi_{\xi\xi} - \phi + \frac{\kappa(c-\kappa)}{c-\phi} = 0.
\end{align}
Multiplying it by $\phi_\xi$ and integrating with respect to $\xi$ gives the total energy $E:\RR^2\to\RR$,
\begin{align}\label{e:energy}
E(\phi,\phi_\xi) := \frac 1 2 (\phi_\xi)^2 + V(\phi) = E_\hom,
\end{align}
where the potential energy $V(\phi):= -\phi^2/2 - \kappa(c-\kappa)\ln(c-\phi)$, and the integration constant $E_\hom$ corresponding to the solitary waves, which is given by
\begin{align}\label{e:Ehom}
E_\hom := E(\kappa,0) = V(\kappa) = -\frac 1 2 \kappa^2 - \kappa(c-\kappa)\ln(c-\kappa).
\end{align}
We note that the real system requires the conditions $\phi<c$ and $\kappa<c$ due to the logarithm in the potential energy.
Concerning the spatial dynamics, we rewrite the second-order ODE \cref{e:ODE2} as a first-order system, which can also be expressed in the Hamiltonian form, namely
\begin{align}\label{e:ODE1}
\frac{\dif}{\dif\xi} \begin{pmatrix} \phi\\ \phi_\xi \end{pmatrix}
= \begin{pmatrix}
\phi_\xi \\ \phi - \kappa(c-\kappa)/(c-\phi)
\end{pmatrix}
= \begin{pmatrix}
0 & 1 \\ -1 & 0
\end{pmatrix}
\nabla E(\phi,\phi_\xi),
\end{align}
where the Hamiltonian (first integral) $E$ is defined as in \cref{e:energy}.

We next give the existence of smooth solitary wave solutions to \cref{e:CH} with nonzero asymoptote $\kappa$ in the following.

\begin{theorem}[Existence of solitary waves for $b=1$]\label{t:existphi}
For fixed $0<\kappa<c/2$, there exists a smooth solitary wave solution to \cref{e:CH} of the form \cref{e:sw} with wave shape $\phi = \phi(\cdot;c,\kappa)\in\Cspace^\infty(\RR)$ and phase variable $\xi=x-ct$ that satisfying $\phi(\xi;c,\kappa)\to\kappa$ exponentially as $\xi\to\pm\infty$ and $\partial_\xi\phi(0;c,\kappa) = 0$ up to spatial translations. In addition, $\kappa<\phi(\xi;c,\kappa)<c$ for all $\xi\in\RR$.
\end{theorem}

\begin{proof}
For any fixed $(c,\kappa)$ satisfying $0<\kappa<c/2$, the graph of the potential energy $V(\phi)$ has a local maximum at $\phi=\kappa$ and a local minimum at $\phi=c-\kappa$, and $V(\phi)$ diverges as $\phi$ approaches $c$.
These imply, in $(\phi,\phi_\xi)$-plane, that the system \cref{e:ODE1} has a saddle point at $(\kappa,0)$, a center at $(c-\kappa,0)$, and a singular line at $\phi=c$. Since the energy $E$ from \cref{e:energy} is conserved along orbits, the unstable manifold at $(\kappa,0)$ must intersect the stable manifold, yielding a homoclinic orbit, on which the total energy has the value $E_\hom$, cf.,~\cref{f:HOMc2k0p4}.
As mentioned above, for $E$ to be real-valued we require the condition $\phi<c$, and thus the homoclinic orbit does not cross the singular line from left to right.
Hence, the traveling wave solution associated with the homoclinic orbit with energy $E_\hom$ forms the solitary wave solution $\phi=\phi(\xi;c,\kappa)$, which decays exponentially to $\kappa$ as $\xi\to\pm\infty$ due to the saddleness, and has the range $\kappa<\phi(\xi;c,\kappa)< c$ for all $\xi\in\RR$. The intersection of the homoclinic orbit and the horizontal axis corresponds to $\partial_\xi\phi(0;c,\kappa) = 0$ up to spatial translations.
\end{proof}

\begin{figure}[t!]
\centering
\subfigure[]{\includegraphics[trim = 5.5cm 8.5cm 6cm 9cm, clip, height=4.5cm]{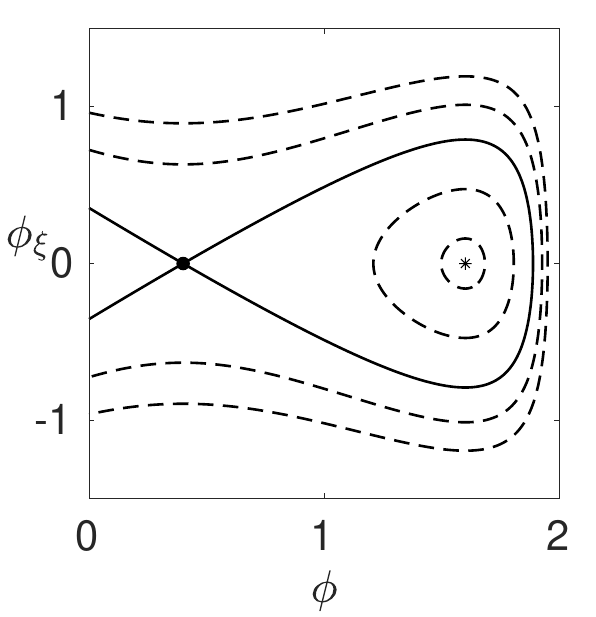}\label{f:HOMc2k0p4}}
\hfil
\subfigure[]{\includegraphics[trim = 5.5cm 8.5cm 6cm 9cm, clip, height=4.5cm]{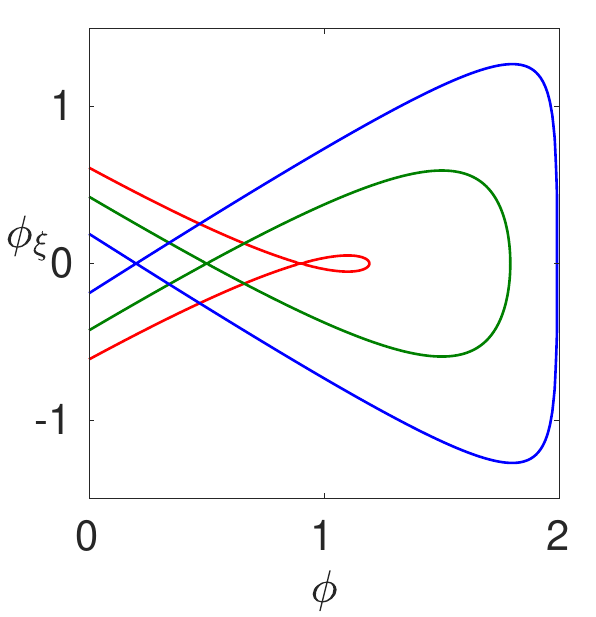}\label{f:HOMc2}}
\hfil
\subfigure[]{\includegraphics[trim = 5.5cm 8.5cm 6cm 9cm, clip, height=4.5cm]{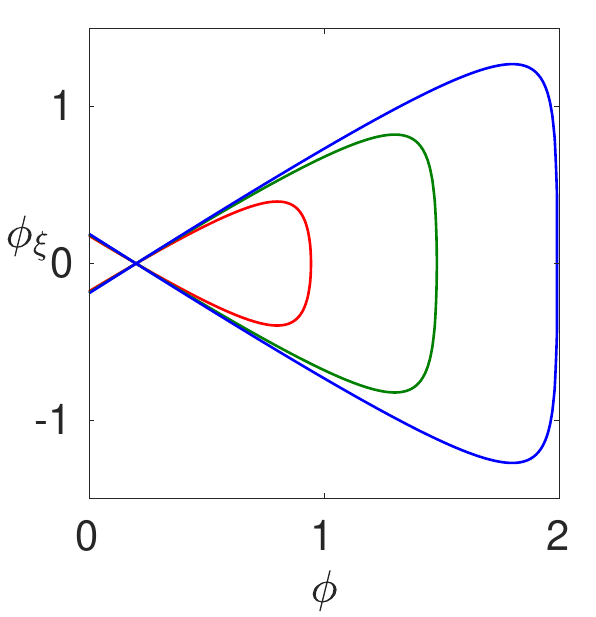}\label{f:HOMk0p2}}
\caption{Samples of phase portraits for the ODE system \cref{e:ODE1}. (a) The homoclinic orbit (solid)
approaches the saddle point (dot) at $(0.4,0)$, with the center (asterisk) at $(1.6,0)$, and the portion of it that lies in $\phi>0.4$ corresponds to the traveling solitary wave solution of \cref{e:CH}; the dashed curves inside the homoclinic orbit are periodic orbits that correspond to spatially periodic traveling wave solutions of \cref{e:CH};
the parameters are $c=2$ and $\kappa=0.4$. Only homoclinic orbits are plotted in (b) for fixed $c=2$, $\kappa=0.9$ (red), $\kappa=0.5$ (green) and $\kappa=0.2$ (blue); and in (c) for fixed $\kappa=0.2$, $c=1$ (red), $c=1.5$ (green) and $c=2$ (blue).}
\label{f:homoclinic}
\end{figure}

\begin{remark}[]\label{r:ratio}
For $\kappa=0$ or $\kappa=c$ the system \cref{e:ODE1} has only a saddle point; for $\kappa = c/2$ \cref{e:ODE1} has a degenerate equilibrium;
for $\kappa>c$ \cref{e:ODE1} has two saddle points. All these cases do not allow the existence of homoclinic orbits, thus no solitary waves. For $c/2 <\kappa<c$ the solitary waves exist, and the proof is analogous to that of \cref{t:existphi}, but $\phi(\xi)$ decays to $c-\kappa$ as $\xi\to\pm\infty$, so we omit the discussion of this case.
\end{remark}

Differentiating \cref{e:integratedeqn} with respect to $\xi$ gives the expressions of $\mu$
and its derivatives in terms of $\phi$ as follows
\begin{align}\label{e:murelation}
\mu = \frac{\kappa(c-\kappa)}{c-\phi},\quad \mu_\xi = \frac{\kappa(c-\kappa)\phi_\xi}{(c-\phi)^2},\quad \mu_{\xi\xi} = \kappa(c-\kappa)\left(\frac{2\phi_\xi^2}{(c-\phi)^3} + \frac{\phi_{\xi\xi}}{(c-\phi)^2}\right).
\end{align}
Combining the above relations and \cref{t:existphi}, we may obtain the following existence result of traveling wave solutions in terms of $\mu$. This is required since we will consider the orbital stability of the traveling wave solutions $\mu$ instead of the solitary waves $\phi$ in \cref{s:stability}.

\begin{corollary}\label{c:existmu}
For fixed $0<\kappa<c/2$, there exists a traveling wave solution to \cref{e:CHham} of the form \cref{e:tw} with wave shape $\mu = \mu(\cdot;c,\kappa)\in\Cspace^\infty(\RR)$ and phase variable $\xi=x-ct$ that satisfying $\mu(\xi;c,\kappa)\to\kappa$ exponentially as $\xi\to\pm\infty$ and $\partial_\xi\mu(0;c,\kappa) = 0$ up to spatial translations. In addition, $\mu = (1-\partial_\xi^2)\phi$ with $\phi$ obtained in \cref{t:existphi}, and $\mu(\xi;c,\kappa)>\kappa$ for all $\xi\in\RR$.
\end{corollary}

As a continuation of existence results and in preparation for the stability analysis presented in \cref{s:stability}, we require the refinement of the range of traveling wave solutions and the parities of them and their derivatives with respect to the spatial variable.
\begin{remark}[Range of $\phi$]\label{r:rangephi}
Since the energy is conserved along the homoclinic orbit, and the existence of the center at $(c-\kappa,0)$ in the phase plane, we deduce that there exists a $(c,\kappa)$-dependent value $G_{c,\kappa}$ that satisfies $G_{c,\kappa}\in(c-\kappa,c)$ and $E(G_{c,\kappa},0)=E_\hom$, such that the solitary wave $\phi(\xi;c,\kappa)$ satisfies
\begin{align*}
\kappa<\phi(\xi;c,\kappa)\leq G_{c,\kappa} < c,\quad \xi\in\RR,
\end{align*}
i.e., $G_{c,\kappa} = \max_{\xi\in\R}(\phi(\xi;c,\kappa))$ for each fixed $(c,\kappa)$.

Analyzing the monotonicity of $G_{c,\kappa}$, cf.,~\cref{s:monotone}, we find that for fixed $c$ the value of $G_{c,\kappa}$ increases to $c$ as $\kappa$ decreases to $0$; for fixed $\kappa$ the value of $G_{c,\kappa}$ increases in $c$.
We plot examples in \cref{f:HOMc2,f:HOMk0p2} that illustrate these.
\end{remark}

\begin{remark}[Range of $\mu$]\label{r:rangemu}
We recall the relations in \cref{e:murelation}, and denote $\mu = \mu(\phi)$ with abuse of notation.
Since $\mu(\kappa) = \kappa$ and $\mu(\phi)$ monotonically increases in $\phi$ and diverges as $\phi\to c$, it satisfies
\begin{align*}
\kappa<\mu(\xi;c,\kappa)\leq M_{c,\kappa} <\infty,\quad \xi\in\RR,
\end{align*}
where $M_{c,\kappa}:=\kappa(c-\kappa)/(c-G_{c,\kappa})$. Rearranging $E(G_{c,\kappa},0) = E_\hom$ as mentioned in \cref{r:rangemu}, yields
\begin{align*}
\frac{c-\kappa}{c-G_{c,\kappa}} = \exp\left(\frac{G_{c,\kappa}^2-\kappa^2}{2\kappa(c-\kappa)}\right)\quad \Rightarrow \quad M_{c,\kappa} = \kappa\exp\left(\frac{G_{c,\kappa}^2-\kappa^2}{2\kappa(c-\kappa)}\right).
\end{align*} Hence, for fixed $c$ the value of $M_{c,\kappa}$ grows exponentially as $\kappa$ decreases to $0$; for fixed $\kappa$ the value of $M_{c,\kappa}$ grows exponentially in $c$ since $(c-\kappa)^2<G_{c,\kappa}^2<c^2$.
Moreover, $\mu(\xi;c,\kappa)$ exponentially decays to $\kappa$ as $\xi\to\pm\infty$.
\end{remark}

\begin{remark}[Parity]\label{r:parity}
Due to the relations in \cref{e:murelation}, the even function $\phi$ implies that $\mu$ and $\mu_{\xi\xi}$ are even functions, and $\mu_\xi$ is odd with mean zero, i.e., $\int_\R \mu_\xi\,\dif\xi = 0$.
\end{remark}

\subsection{Case $0<b<1$}\label{s:existbn1}
We briefly discuss the existence of the smooth solitary waves to \cref{e:bCH} in case of $0<b<1$. We note that its proof is analogous to that of $b>1$, which has been proved in \cite{LP2022}. Hence, we just present the result without any further proofs.

\begin{lemma}[Existence of solitary waves for $0<b<1$, cf.~Theorem 1 in \cite{LP2022}]\label{l:existbn1}
For fixed $0<b<1$ and $0<\kappa<\frac{c}{b+1}$, there exists a smooth solitary wave solution to \cref{e:bCH} of the form \cref{e:sw} with wave shape $\phi = \phi_b(\cdot;c,\kappa)\in\Cspace^\infty(\RR)$ and phase variable $\xi=x-ct$ that satisfying $\phi_b(\xi;c,\kappa)\to\kappa$ exponentially as $\xi\to\pm\infty$ and $\partial_\xi\phi_b(0;c,\kappa) = 0$ up to spatial translations. In addition, $\kappa<\phi_b(\xi;c,\kappa)<c$ for all $\xi\in\RR$.
\end{lemma}

Note that the solitary wave solutions and the admissible range of $\kappa$ depend on the parameter $b$. We recall, in contrast to \cref{e:energy} for $b=1$, that the total energy for \cref{e:bCH} with $b>1$ has been given in \cite{LP2022}, which takes the form
\begin{align}\label{e:hombneq1}
E(\phi,\phi_\xi) = \frac 1 2 (b-1)(\phi_\xi^2 - \phi^2) + \frac{\kappa(c-\kappa)^b}{(c-\phi)^{b-1}} = c\kappa - \frac 1 2 (b+1)\kappa^2.
\end{align}
We find that this is also the total energy for the case $0<b<1$, and the homoclinic orbit associated with the smooth solitary waves to \cref{e:bCH} can be generated by \cref{e:hombneq1} and is similar to that in \cref{f:HOMc2k0p4}. We illustrate this for different values of $b$ in \cref{f:homovaryb}.

\begin{figure}[t!]
\centering
\includegraphics[trim = 5.5cm 8.5cm 6cm 9cm, clip, height=4.5cm]{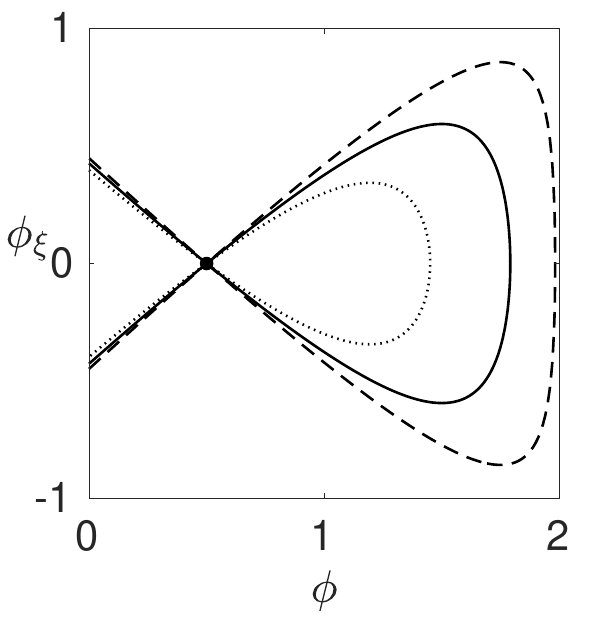}
\caption{Samples of homoclinic orbits for the ODE system \cref{e:ODE1}, which associate with the smooth solitary waves to \cref{e:bCH} with $b=0.7$ (dashed), $b=1$ (solid) and $b=1.4$ (dotted), and fixed parameters $\kappa=0.5$ and $c=2$.}
\label{f:homovaryb}
\end{figure}

\tikzset{
   flow/.style =
   {decoration = {markings, mark=at position #1 with {\arrow{>}}},
    postaction = {decorate}
   }}

\tikzset{global scale/.style={
    scale=#1,
    every node/.append style={scale=#1}
  }
}

\section{Stability of solitary waves}\label{s:stability}
In this section, we consider the orbital stability of the smooth solitary waves to \cref{e:bCH} for the case of $b=1$ only, i.e., \cref{e:CH}.
Our proof is essentially based upon a general framework, motivated by the approach of \cite{GSS1987}, which characterizes the critical points of Hamiltonian systems with symmetries and conserved quantities via the analysis of a constrained operator. We also refer the readers to, e.g.,~\cite{KP2015}, for a detailed discussion on KdV equations and extension to the general framework.

\begin{remark}[Case $0<b<1$]\label{r:stability}
As we mentioned in \cref{s:Ham}, the $b$-CH equation \cref{e:bCH} for both $0<b<1$ and $b>1$ share the same Hamiltonian form and conserved quantities. This leads to the same Lagrangian functional (up to a minus sign) and stability criterion as in \cite{LP2022}. Hence, the proof of the stability of the solitary waves in \cref{e:bCH} for $0<b<1$ follows the papers \cite{LP2022,LL2023}, and we do not pursue this further in the present paper.
\end{remark}

Returning to the case $b=1$, we recall that \cref{e:m} can be written in the Hamiltonian form \cref{e:CHham} in terms of $m$, namely
\begin{align*}
\partial_t m = \Jm \frac{\delta\calH}{\delta m}(m),
\end{align*}
where the operator $\Jm$ defined in \cref{e:J} maps
\begin{align*}
\Jm:= -\partial_x\big(m(1-\partial_x^2)^{-1}\partial_x^{-1}(m\partial_x\cdot)\big): \Yspace\to \Xspace,
\end{align*}
and we take
\begin{align}\label{e:fctspace}
\Yspace = \Hspace^1(\RR)\subset\Lspace^2(\RR) = \Xspace.
\end{align}
Such selection of function spaces is appropriate and we demonstrate this in \cref{s:opJ}. This is a skew-symmetric operator with respect to the $\Lspace^2$-inner product $\langle u,v \rangle_{\Lspace^2} = \int_\RR u\bar v\,\dif x$, i.e., $\langle \Jm u, v\rangle_{\Lspace^2} = -\langle u,\Jm v \rangle_{\Lspace^2}$ for all $u,v\in\Lspace^2(\RR)$.
The function $\delta\calH/\delta m$ denotes the variational derivative of the functional $\calH$ given in \cref{e:H}, and $(\delta\calH/\delta m)(m) = \ln m - \ln\kappa$.
When there is no ambiguity, in the remainder of this section we denote $\langle \cdot,\cdot \rangle_{\Lspace^2}$ as $\langle \cdot,\cdot \rangle$.

In the moving frame $\xi:=x-ct$, we seek the solution of the form $m(x,t) = m(\xi,t)$, with abuse of notations, that satisfies \cref{e:CHham}, namely
\begin{align}\label{e:mJc}
\partial_t m = \Jm\frac{\delta\calH}{\delta m}(m) + c\partial_\xi m = \Jm\left(\frac{\delta\calH}{\delta m}(m) + c\frac{\delta \calQ}{\delta m}(m)\right),
\end{align}
where the form of $\Jm$ gives a second conserved quantity: the charge $\calQ$, which is defined via
\begin{align}\label{e:JdQ}
\Jm\frac{\delta\calQ}{\delta m}(m) = \partial_\xi m.
\end{align}
We remark that the operator $\Jm$ has no kernel on $\Lspace^2(\RR)$, and we give a proof in \cref{s:nokernel}. Hence, we may solve the variational derivative of $\calQ$ from \cref{e:JdQ} by inverting $\Jm$, i.e., formally $(\delta\calQ/\delta m)(m) = \Jm^{-1}\partial_\xi m$, and this gives the charge of the form
\begin{align}\label{e:charge}
\calQ(m) =-\frac 1 2\kappa^{-1} \calQ_1(m) - \frac 1 2 \kappa \calQ_2(m),
\end{align}
where $\calQ_1$ and $\calQ_2$ are given in \cref{e:Qj}. The choice of the prefactors in \cref{e:charge} arises from the selection of the function spaces \cref{e:fctspace}, see more details of the computation in \cref{s:Casimirs}.
In addition, we note that $\calQ(m)$ is time-invariant since $\calH(m(\xi))$ vanishes at $\xi=\pm\infty$. Indeed,
\begin{align*}
\partial_t\calQ(m) &= \left\langle \partial_t m, \frac{\delta\calQ}{\delta m}(m) \right\rangle = \left\langle \Jm \frac{\delta\calH}{\delta m}(m), \frac{\delta\calQ}{\delta m}(m) \right\rangle\\
& = - \left\langle \frac{\delta\calH}{\delta m}(m), \Jm\frac{\delta\calQ}{\delta m}(m) \right\rangle = - \left\langle \frac{\delta\calH}{\delta m}(m), \partial_\xi m \right\rangle = - \int_\RR \partial_\xi \calH(m)\,\dif\xi = 0.
\end{align*}
Motivated by \cref{e:mJc}, we define the Lagrangian functional for \cref{e:CHham}
\begin{align}\label{e:Lagrangian}
\LG(m) = -\calH(m) - c\calQ(m)  = -\calH(m) + \frac 1 2 c\kappa^{-1} \calQ_1(m) + \frac 1 2 c\kappa \calQ_2(m).
\end{align}

\begin{remark}
In comparison with the Lagrangian for the case $b\neq1$ in \cite[Eq.~(32)]{LP2022}, neither the conserved functionals nor the Lagrangian multipliers in \cref{e:Lagrangian} are the limits of those in \cite[Eq.~(32)]{LP2022} as $b\to1$; in particular, the multipliers in \cite[Eq.~(32)]{LP2022} are  $(c,\kappa)$-independent constants as $b\to1$. In other words, \cref{e:Lagrangian} is degenerate and is not the continuation of \cite[Eq.~(32)]{LP2022} as $b\to 1$.
\end{remark}

Hence, in the moving frame, \cref{e:mJc} can be written as the form
\begin{align}\label{e:CHtravel}
\partial_t m = - \Jm \frac{\delta\LG}{\delta m}(m),
\end{align}
where the variational derivative of $\LG$ takes the form
\begin{align*}
\frac{\delta\LG}{\delta m}(m) = \ln\kappa - \ln m + \frac 1 2 c \kappa^{-1} + \frac 1 2 c \kappa (-m^{-2} + 3m^{-4}m_\xi^2 - 2m^{-3}m_{\xi\xi}).
\end{align*}
We consider the traveling wave solutions of the form $m(\xi,t) = \mu(\xi)$, which is stationary in the moving frame. Substitution into \cref{e:CHtravel} gives the traveling equilibrium equation
\begin{align}\label{e:steadyLambda}
\calJ_\mu \frac{\delta\LG}{\delta m}(\mu) = 0 \quad\Rightarrow\quad \frac{\delta\LG}{\delta m}(\mu) = 0,
\end{align}
where the latter arises from the empty kernel of $\calJ_\mu$ on $\Lspace^2(\RR)$.
This means the equilibrium solution $\mu$ is the critical point of $\LG$.
Differentiating \cref{e:steadyLambda} with respect to $\xi$, yields
\begin{align}\label{e:kernel}
\frac{\delta^2\LG}{\delta m^2}(\mu)(\partial_\xi\mu) = 0.
\end{align}
We denote $\calL$ the second variational derivative of $\LG$, which takes the form
\begin{align}\label{e:L}
\calL:=\frac{\delta^2\LG}{\delta m^2}(\mu) = - c\kappa\left(\partial_\xi(\mu^{-3}\partial_\xi\cdot) - \mu^{-3} + 6\mu^{-5}\mu_\xi^2 - 3\mu^{-4}\mu_{\xi\xi} + \frac{1}{c\kappa \mu}\right).
\end{align}

\begin{lemma}\label{l:opL}
The operator $\calL:\Hspace^2(\RR)\to\Lspace^2(\RR)$ defined in \cref{e:L} is a self-adjoint operator with respect to the $L^2$-inner product, and its spectrum is composed of
\begin{itemize}
\item [(a)] a simple zero eigenvalue, with the associated eigenfunction $\partial_\xi\mu$, i.e., $\partial_\xi\mu\in\ker(\calL)$;
\item [(b)] a simple negative eigenvalue $\lambda_0$, with the associated eigenfunction $\psi_0$ that has no zeros;
\item [(c)] a finite number of simple positive eigenvalues $\lambda_j$, $j=2,\dots,N$ that lie in the interval $(0,(c-\kappa)/\kappa^2)$;
\item [(d)] the essential spectrum $\sigma_\ess(\calL) = \{\lambda\in\CC: \lambda= (ck^2 + c-\kappa)/\kappa^2,\,k\in\RR\}$, which is strictly positive.
\end{itemize}
\end{lemma}

\begin{proof}
The divergence form of the differential operator $\calL$ implies its self-adjointness.

(a) It is straightforward from \cref{e:kernel}.

(b) We know, from \cref{c:existmu}, that the kernel eigenfunction $\partial_\xi\mu$ has only one simple zero, hence a Sturm-Liouville theorem (cf.~\cite[Theorem 2.3.3]{KP2015}) implies the statement.

(c) The asymptotic operator $\calL_\infty$ associated with the exponentially asymptotic operator $\calL$, i.e., replacing the coefficients in $\calL$ by their limiting values at $\xi=\pm\infty$, takes the form
\begin{align*}
\calL_\infty = \frac{c}{\kappa^2}\left(-\partial_\xi^2 + 1 - \frac{\kappa}{c}\right).
\end{align*}
From the aforementioned Sturm-Liouville theorem, there exists an upper bound of the point spectrum, whose value is given by the constant term of $\calL_\infty$.

(d) The essential spectrum of $\calL_\infty$ is given by $\{\lambda\in\CC: \lambda = (ck^2 + c-\kappa)/\kappa^2,\,k\in\RR\}$, which is strictly positive due to the condition $c>\kappa$ (cf.,~\cref{t:existphi}). Since $\calL$ is a relatively compact perturbation of $\calL_\infty$, from the Weyl essential spectrum theorem, it follows that $\calL$ and $\calL_\infty$ have the same essential spectra, cf., e.g.,~\cite[Theorem 2.2.6 \& Theorem 3.1.11]{KP2015}.
\end{proof}

\paragraph{Equilibrium solutions} Since \cref{e:CHtravel} has a translation symmetry, we denote the manifold
\begin{align}\label{e:MT}
\calM_\calT:=\{\calT(s)\mu: s\in\RR\}
\end{align}
the set of all the translations of the equilibrium solution $\mu$ to \cref{e:CHtravel}, where $\calT(s)$ is the translation operator, i.e., $(\calT(s)\mu)(\xi):= \mu(\xi+s)$.
We can find a foliation of a neighborhood $B$ of $\calM_\calT$ for which the solution $m=m(t)\in B$ to \cref{e:CHtravel} can be uniquely decomposed as
\begin{align}\label{e:decompose}
m(t) = \calT(s(t))\mu + h(t),\quad \langle h(t),\calT(s(t))\partial_\xi\mu \rangle = 0,
\end{align}
where the small perturbation $h(t)$ is orthogonal to the kernel of the adjoint operator of $\calL$, which is spanned by $\partial_\xi\mu$ due to the self-adjointness of $\calL$, cf.,~\cref{l:opL}, and $s(t)$ is a continuous function such that the decomposition \cref{e:decompose} holds for $t\in[0,t_1)$ for some $t_1>0$. We refer the readers to \cite[Lemma 4.3.3]{KP2015} for more details.

\medskip
For readability, we require the definition of the orbital stability of solitary waves $\phi$ obtained in \cref{t:existphi} to \cref{e:CH}. Equivalently, and for convenience, we define the orbital stability in terms of the traveling wave solution $m = \mu = (1-\partial_\xi^2)\phi$ to \cref{e:CHham} in the moving frame $\xi=x-ct$.
\begin{definition}[Orbitally stable]
Let $m(x,t) = \mu(x-ct)$ be a traveling wave solution to \cref{e:CHham} with $\mu\in\calX_\kappa$, which has a translation symmetry $\calT=\calT(s)$. The manifold $\calM_\calT$ defined in \cref{e:MT} is orbitally stable in $\calX_\kappa$ if for any $\eps>0$ there exists $\delta>0$ such that for any $m_0\in\calX_\kappa$ satisfying $\|m_0-\mu\|_{\Hspace^1}< \delta$, the solution $m(\cdot,t)\in\calX_\kappa$ of \cref{e:CHham} with $m(\cdot,0)=m_0$ satisfies
\[
\inf_{s\in\RR}\|m(\cdot,t)-\calT(s)\mu\|_{\Hspace^1}< \eps,\quad t>0.
\]
\end{definition}

We will show that the $\Hspace^1(\RR)$-norm of the perturbation
\begin{align*}
h(\xi,t) := m(\xi,t) - \calT(s(t))\mu(\xi)
\end{align*}
remains uniformly small for all $t>0$, and we expect that it can be controlled by the difference between $\Lambda(m)$ and $\Lambda(\calT(s)\mu)$, i.e., $\|h\|_{\Hspace^1}^2\leq \alpha |\LG(m) - \LG(\calT(s)\mu)|$ for some $\alpha>0$. This motivates us to define
the time-invariant quantity
\begin{align}\label{e:LGdif}
\overline\LG := \LG(m) - \LG(\calT(s)\mu).
\end{align}
For notational convenience, we assume that $s\equiv0$, so that $\calT(s)=\id$.
Due to the validity of the decomposition \cref{e:decompose}, we can expand the Lagrangian $\LG(m)$ near $m=\mu$
\begin{align}\label{e:expandLambda}
\LG(\mu+h) = \LG(\mu) + \left\langle \frac{\delta\LG}{\delta m}(\mu), h \right\rangle + \frac 1 2 \left\langle \frac{\delta^2\LG}{\delta m^2}(\mu)h, h \right\rangle +  \calR,
\end{align}
where the remainder term $\calR$ satisfies the estimate
\begin{align}\label{e:remainder}
|\calR| \leq C\|h\|_{\Hspace^1}^3,
\end{align}
for some $C>0$. Since this is in some sense standard, we defer the proof in \cref{s:estimate}.
Hence the expansion of $\overline\LG$ in powers of $h$ yields
\begin{align}\label{e:expandLG}
\overline\LG = \frac 1 2 \langle \calL h, h \rangle + \calO(\|h\|_{\Hspace^1}^3).
\end{align}
\cref{l:opL} implies that the bilinear form $\langle \calL h,h \rangle$ is not coercive since $\calL$ has non-positive eigenvalues.
In order to obtain the coercivity of $\calL$, we must consider that such bilinear form acts on a constrained function space.

Let $m=m(t)$ be the solution of \cref{e:CHtravel} with the initial datum $m(0)=m_0$ such that $\calQ(m_0)=\calQ(\mu)$.
Since $\calQ$ is time-invariant, we have the constraint $\calQ(m(t)) = \calQ(\mu)$ for all $t\geq0$.
This constraint and \cref{e:decompose} imply that the perturbation $h(t)$ lies in the nonlinear admissible space
\begin{align*}
\calA_c := \{h\in\Hspace^1(\RR): \calQ(\mu + h)=\calQ(\mu),\, \langle h,\mu_\xi \rangle = 0\}.
\end{align*}
By triangle inequality and $\Hspace^1$ continuity of $\calQ(\mu(c))$ with respect to $c$,  we may assume without loss of generality that $\calQ(\mu+h_0) = \calQ(\mu)$ with the initial perturbation $h_0:=h(0)$.


Since we consider small perturbations $h$, it is sufficient to consider $h$ in the tangent space of $\calA_c$ at $h=0$. Taylor expanding $\calQ$
yields
\begin{align}\label{e:expandQ}
\calQ(\mu+h) -\calQ(\mu) = \left\langle \frac{\delta \calQ}{\delta m}(\mu), h \right\rangle + \calO(\|h\|^2_{\Hspace^1}) =  \langle \psi_\calQ, h \rangle + \calO(\|h\|^2_{\Hspace^1}),
\end{align}
where the variational derivative of $\calQ$ is given by
\begin{align}\label{e:dQ}
\psi_\calQ :=
\frac{\delta \calQ}{\delta m}(\mu) =  -\frac 1 2 \kappa(-\mu^{-2} + 3\mu^{-4}\mu_\xi^2 - 2\mu^{-3}\mu_{\xi\xi}) - \frac 1 2 \kappa^{-1}.
\end{align}
\begin{remark}\label{r:sign}
Substituting \cref{e:murelation} into \cref{e:dQ} and combining the relation \cref{e:ODE2},
yields
\begin{align}\label{e:psiQ}
\psi_\calQ =\frac{(\phi_\xi)^2 - \phi^2 + \kappa^2}{2\kappa(c-\kappa)^2}  = \frac{1}{c-\kappa}\ln\left(\frac{c-\phi}{c-\kappa}\right),
\end{align}
where the second equality arises from \cref{e:energy}, and $\psi_\Q$
is strictly negative for $\kappa<\phi<c$, and thus $\psi_\calQ = (\delta\calQ/\delta m)(\mu(\xi)) <0$ for all $\xi\in\RR$.
\end{remark}

In the tangent space of $\calA_c$, the difference $\calQ(\mu+h) -\calQ(\mu)$ is approximated by the leading order term $\langle \psi_\calQ, h \rangle$. Hence, the linearized version of the nonlinear constraint in $\calA_c$ is the condition $\langle \psi_\calQ, h \rangle=0$, which motivates the following linear admissible space
\begin{align*}
\calA_c' := \{h\in\Hspace^1(\RR): \langle h,\mu_\xi \rangle = \langle h, \psi_\calQ \rangle = 0\}.
\end{align*}
We next characterize the spectrum of $\calL$ induced by the bilinear form $\langle \calL u,v \rangle$ that is constrained on the space $\calA_c'$.

\begin{lemma}[Coercivity]\label{l:coercivity}
If the functional $\calQ(\mu)$ is decreasing in the wave speed $c$, namely
\begin{align}\label{e:dcQ}
\frac{\dif}{\dif c}\calQ(\mu) < 0,
\end{align}
then there exists $\alpha>0$ such that
\begin{align}\label{e:H1coercivity}
\langle \calL h, h \rangle \geq \alpha\|h\|_{\Hspace^1}^2, \quad h\in\calA_c'.
\end{align}
\end{lemma}
The proof of this lemma is somewhat standard, which is analogous to that for the KdV equation, cf., e.g.,~\cite[Lemma 5.2.3]{KP2015}, with some adaptation due to the different forms of the charge $\calQ$ and the linear operator $\calL$. For completeness, we defer the proof in \cref{s:coercivity}.
\medskip
As a continuation of the $\Hspace^1$-coercivity of $\calL$, and in preparation for the orbital stability of solitary waves, we require the estimate of the bilinear form $\langle \calL h, h \rangle$ for $h\in\calA_c$.

\begin{lemma}\label{l:Aest}
Under the condition \cref{e:dcQ}, the estimate
\begin{align}\label{e:Aest}
\frac 1 2 \langle \calL h, h \rangle \geq \alpha\|h\|_{\Hspace^1}^2 - \beta\|h\|_{\Hspace^1}^3,
\end{align}
holds for $h\in\calA_c$ with some $\alpha,\beta>0$.
\end{lemma}

\begin{proof}
The decomposition \cref{e:decomposeY} defines the projection $P:\Yspace_0\to{\rm span}\{\psi_\calQ\}$
\begin{align*}
Pu = \frac{\langle u,\psi_\calQ \rangle}{\langle \psi_\calQ, \psi_\calQ \rangle}\psi_\calQ,
\end{align*}
and its complement $\Pi=\id-P$ defined in \cref{e:projPi},
which can split any perturbation $h\in\calA_c\subset\Yspace_0$ into
\begin{align}\label{e:decomposeh}
h = h_1 + \nu\psi_\calQ
\end{align}
with $h_1 = (\id-P)h = \Pi h \in\calA_c'$ and the coefficient $\nu = \langle h,\psi_\calQ \rangle/\langle \psi_\calQ, \psi_\calQ \rangle$. The nonlinear constraint in $\calA_c$ and Taylor expansion \cref{e:expandQ} indicate that
\begin{align*}
\calQ(\mu+h) - \calQ(\mu) &= \langle \psi_\calQ, h \rangle + \calO(\|h\|_{\Hspace^1}^2)\\
& = \langle \psi_\calQ, h_1 \rangle + \nu\langle \psi_\calQ, \psi_\calQ \rangle + \calO(\|h\|_{\Hspace^1}^2) = 0.
\end{align*}
Since $\langle \psi_\calQ, h_1 \rangle = 0$ for $h_1\in\calA_c'$, the coefficient $\nu$ is of order $\calO(\|h\|_{\Hspace^1}^2)$, and thus $h_1 = \calO(\|h\|_{\Hspace^1})$ due to \cref{e:decomposeh}. If the condition \cref{e:dcQ} is satisfied, then the $\Hspace^1$-coercivity \cref{e:H1coercivity} holds for $h_1\in\calA_c'$ with some $\alpha_1>0$. This implies the following estimate over $\calA_c$
\begin{align*}
\langle \calL h, h \rangle &= \langle \calL (h_1 + \nu\psi_\calQ), h_1 + \nu\psi_\calQ \rangle = \langle \calL h_1, h_1 \rangle + \calO(\|h\|_{\Hspace^1}^3) \geq \alpha_1\|h_1\|_{\Hspace^1}^2 + \calO(\|h\|_{\Hspace^1}^3).
\end{align*}
Hence,
\cref{e:Aest} holds by choosing
$\alpha = \alpha_1\|h_1\|_{\Hspace^1}^2/(2\|h\|_{\Hspace^1}^2)$ and sufficiently large $\beta>0$.
\end{proof}

\begin{lemma}\label{l:Qeq}
The functional $\calQ(\mu)$ is decreasing in $c$, i.e., \cref{e:dcQ}, if and only if
\begin{equation}\label{e:Qincr}
\frac{\dif}{\dif c}\Q(\phi,c)>0, \quad \Q(\phi,c):=\int_{\mathbb{R}}\left(\frac{c-\kappa}{c-\phi}\ln\left(\frac{c-\kappa}{c-\phi}\right)
-\frac{c-\kappa}{c-\phi}+1\right)\dif\xi,
\end{equation}
where $\phi=\phi(\xi;c,\kappa)$ is the solitary wave solution obtained in \cref{t:existphi} which relates $\mu$ through \cref{e:integratedeqn}.
\end{lemma}

\begin{proof}
Making use of the relation \cref{e:murelation} and the expression \cref{e:psiQ}, we have
\begin{align*}
\frac{\dif}{\dif c}\calQ(\mu)\ &=\ \langle \psi_\calQ, \partial_c\mu \rangle =\ \int_\RR \left(\frac{1}{c-\kappa}\ln\left(\frac{c-\phi}{c-\kappa}\right)\frac{\dif}{\dif c} \left(\frac{\kappa(c-\kappa)}{c-\phi}\right)\right)\dif\xi\\
&=\ -\frac{\kappa}{c-\kappa}\frac{\dif}{\dif c}\int_\RR \left(\frac{c-\kappa}{c-\phi}\ln\left(\frac{c-\kappa}{c-\phi}\right) - \frac{c-\kappa}{c-\phi}+1\right)\dif\xi =\ -\frac{\kappa}{c-\kappa}\frac{\dif}{\dif c}\Q(\phi,c),
\end{align*}
where the constant $1$ in the integrand of the second integral guarantees the integrability since $\phi\to\kappa$ as $\xi\to\pm\infty$.
Since $c>\kappa$ is required for the existence of solitary waves (cf.,~\cref{t:existphi}), the statement is proved.
\end{proof}

\begin{theorem}\label{t:stable}
For $b=1$ and $\kappa\in(0, c/2)$, the manifold $\calM_\calT$ defined in \cref{e:MT} is orbitally stable in $\calX_\kappa$ as defined in \cref{e:mspace} if \cref{e:Qincr} holds.
\end{theorem}

\begin{proof}
Since \cref{e:Qincr} holds, \cref{l:coercivity,l:Aest,l:Qeq} imply that
\begin{align*}
|\overline\LG|\geq \overline\LG \geq \alpha\|h\|_{\Hspace^1}^2 - \tilde\beta\|h\|_{\Hspace^1}^3,
\end{align*}
for some $\tilde\beta>0$. The remainder is to prove that the $\Hspace^1$-norm of the small perturbation $h$ can be controlled by $|\overline\LG|$, which is analogous to that for the KdV equation, cf., e.g., the proof of \cite[Theorem 5.2.2]{KP2015}, so we omit it here.
\end{proof}

\section{Proof of stability criterion \cref{e:Qincr}}\label{s:criterion}
In this section, we give a proof of the stability criterion \cref{e:Qincr}.
To this end, we need the following preparations that will be used throughout this section and some transformations on the trajectories corresponding to the smooth solitary waves (i.e., the homoclinic orbits).

We first rescale the wave shape of the solitary wave solutions as
\begin{align}\label{e:rescaledSW}
\phi=\kappa+\gamma\varphi,\quad \gamma:=c-\kappa,
\end{align}
where $\gamma\in(c/2,c)$ due to $\kappa\in(0,c/2)$, and $\varphi\in(0,1)$ due to $\phi\in(\kappa,c)$ as in \cref{t:existphi}. In particular, we have
\[
\vphi\in(0,\vphi_0], \quad \vphi_0:=(G_{c,\kappa}-\kappa)/\gamma,
\]
where $G_{c,\kappa}$ is defined as in \cref{r:rangephi}, and we remark that $\vphi_0\to 1$ as $\kappa\to0$ or $c\to\infty$, cf., illustrations in \cref{f:HOMc2,f:HOMk0p2}). Substituting \cref{e:rescaledSW} into the first-order system \cref{e:ODE1}, yields
\begin{equation}\label{e:rescale1stODE}
\frac{\dif\varphi}{\dif\xi}=\psi,\quad
  \frac{\dif\psi}{\dif\xi}=  \vphi - \frac{\kappa\vphi}{\gamma(1-\vphi)},
\end{equation}
where $\psi:=\varphi_\xi$, and the first integral is given by
\begin{equation}\label{e:tildeH}
\widetilde H(\varphi,\psi) :=\gamma(\psi^2-\varphi^2)-2\kappa\varphi-2\kappa\ln(1-\varphi)=\widetilde H_0.
\end{equation}
System \cref{e:rescale1stODE} possesses two equilibria, i.e., the saddle point $(0,0)$ and the center $(1-\kappa/\gamma,0)$, where $0<1-\kappa/\gamma<1$, and a singular line $\varphi=1$.
Since the energy $\widetilde H$ is invariant along the homoclinic orbit and equals to the value at the saddle point, the homoclinic orbit of \cref{e:rescale1stODE} possesses the energy $\widetilde H_0=0$; we illustrate these in \cref{f:homtransa}. Moreover, we call $(\vphi_0,0)$ the \emph{turning point}, which is the rightmost intersection point of the homoclinic orbit and horizontal axis in $(\vphi,\psi)$-plane.

\medskip

We next transform the homoclinic orbit of \cref{e:rescale1stODE} as follows. We reformulate the Hamiltonian \cref{e:tildeH} with the energy $\widetilde H_0 = 0$ as
\begin{equation}\label{e:sepH}
\frac{\psi^2}{\varphi+\ln(1-\varphi)}-\frac{\varphi^2}{\varphi+\ln(1-\varphi)}=\frac{2\kappa}{\gamma},
\end{equation}
where $\varphi+\ln(1-\varphi)<0$ for $\varphi\in(0,1)$. To take advantage of the framework for the monotonicity of the period function as introduced in \cref{s:intro},
we define a new variable $\bpsi$ as
\begin{equation}\label{e:trans}
\bpsi :=\frac{\psi}{\sqrt{-\varphi-\ln(1-\varphi)}}.
\end{equation}
Hence, \cref{e:sepH} can be written as the separate form (in terms of $\varphi,\bpsi,\gamma$)
\begin{equation}\label{e:sepHbar}
H(\varphi,\bpsi) := -\bpsi^2 - \frac{\varphi^2}{\varphi+\ln(1-\varphi)} = h,\quad h:=\frac{2\kappa}{\gamma},
\end{equation}
where $h\in(0,2)$ since $0<\kappa<c/2$, and we denote by
\begin{align}\label{e:Gammah}
\Gammah :=\{(\varphi,\bpsi): H(\varphi,\bpsi) = h\}
\end{align}
such trajectory, see \cref{f:homtransb}. We remark that the level curve $\Gammah$ corresponds to the homoclinic orbit of \cref{e:rescale1stODE} with the energy $\widetilde H_0=0$ from \cref{e:tildeH}.
In fact, this is one of the curves of the first integral of the following Hamiltonian system
\begin{equation}\label{e:ex3}
\frac{\dif\varphi}{\dif\tau}=2\bpsi,\quad
\frac{\dif\bpsi}{\dif\tau}= - \frac{\varphi f(\varphi)}{(1-\varphi)(-\varphi-\ln(1-\varphi))^2},
\end{equation}
where
\[
\dif\tau:=\frac{1}{2}\sqrt{-\varphi-\ln(1-\varphi)}\,\dif\xi,\quad f(\varphi):=\varphi+(1-\varphi)\big(\varphi+2\ln(1-\varphi)\big).
\]
We note that since $\sqrt{-\vphi-\ln(1-\vphi)}$ is strictly positive for $\vphi\in(0,1)$, the change of variables $\tau=\tau(\xi)$ is bijective and thus invertible; and it is easy to know that $f(\varphi)>0$ for $\varphi\in(0,1)$ since $f(0)=0$ and $f'(\varphi)>0$ for $\varphi\in(0,1)$.
We remark, that since $\varphi+\ln(1-\varphi)$ is negative and monotonically decreasing in $\varphi$ for $\varphi\in(0,1)$, the curve of $\Gammah$ has a `hyperbolic' shape, and is to the left of the vertical line $\varphi=1$. The turning point $(\vphi_0,0)$ is invariant under the transform \cref{e:trans}, which can be solved by inserting $\bpsi=0$ into \cref{e:sepHbar}. A significant difference between the first integrals \cref{e:tildeH} and \cref{e:sepHbar} is that $\varphi = 0$ is a singularity of the latter. We find, however, that $-\varphi^2/(\varphi+\ln(1-\varphi))$ approaches $2$ as $\varphi\to0$, hence the intersections of $\Gammah$ and the vertical axis are given by $\bpsi=\pm a$ where $a=a(h):=\sqrt{2-h}>0$. We illustrate these in \cref{f:diagram}.

\tikzset{
   flow/.style =
   {decoration = {markings, mark=at position #1 with {\arrow{>}}},
    postaction = {decorate}
   }}

\tikzset{global scale/.style={
    scale=#1,
    every node/.append style={scale=#1}
  }
}

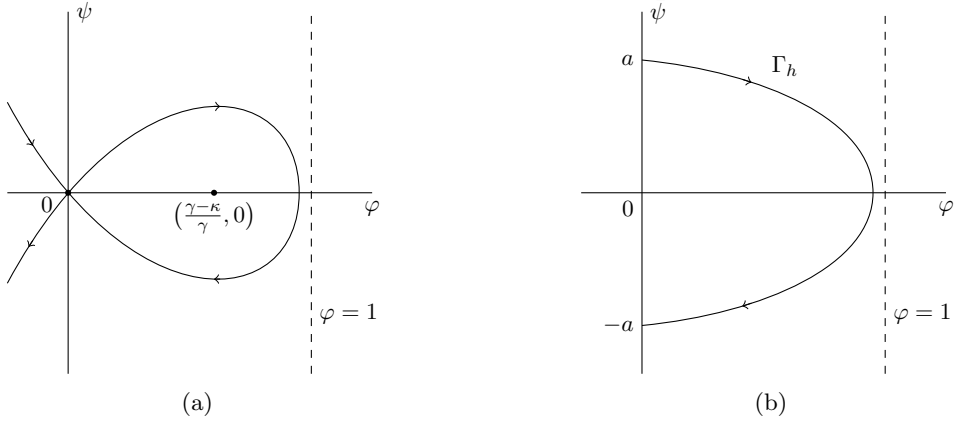
\begin{figure}[t!]
\centering
\subfigure[]{
\begin{tikzpicture}[global scale=0.8]
\draw (-1,0)--(5,0);
\draw (5,0)node[below]{$\varphi$};
\draw (0,-3)--(0,3);
\draw (0,3)node[right]{$\psi$};
\draw[black] (-1,-1.5) .. controls (1,2.25) and (3.8,2).. (3.8,0);
\draw[black] (-1,1.5) .. controls (1,-2.25) and (3.8,-2).. (3.8,0);
\fill[black](0,0)circle(0.05);
\draw[->][black](2.4,1.431)--(2.5,1.43);
\draw[->][black](2.5,-1.43)--(2.4,-1.431);
\draw[->][black](-0.65,0.895)--(-0.58,0.78);
\draw[->][black](-0.58,-0.79)--(-0.65,-0.895);
\draw (-0.35,-0.2)node{$0$};
\fill[black](2.4,0)circle(0.05);
\draw (2.4,0)node[below]{$\big(\frac{\gamma-\kappa}{\gamma},0\big)$};
\draw [dashed](4,-3)--(4,3);
\draw (4,-2) node [right]  {$\varphi=1$};
\end{tikzpicture}
\label{f:homtransa}
}
\hfil
\subfigure[]{
\begin{tikzpicture}[global scale=0.8]
\draw (-1,0)--(5,0);
\draw (5,0)node[below]{$\varphi$};
\draw (0,-3)--(0,3);
\draw (0,3)node[right]{$\bpsi$};
\draw[black] (0,2.2) .. controls (2,2) and (3.8,1.2) .. (3.8,0);
\draw[black] (0,-2.2) .. controls (2,-2) and (3.8,-1.2) .. (3.8,0);
\draw[->][black](1.65,1.88)--(1.8,1.83);
\draw[->][black](1.8,-1.83)--(1.65,-1.88);
\draw (0,0)node[below left]{$0$};
\draw [dashed](4,-3)--(4,3);
\draw (4,-2) node [right]  {$\varphi=1$};
\draw(2,1.8)node[above right]{$\Gammah$};
\draw(0,2.2)node[left]{$a$};
\draw(0,-2.2)node[left]{$-a$};
\end{tikzpicture}
\label{f:homtransb}
}
\caption{
Diagrams of the homoclinic orbit of the rescaled system \cref{e:rescale1stODE} in (a),
and the corresponding trajectory $\Gammah$ defined in \cref{e:Gammah} in (b), where $a=a(h):=\sqrt{2-h}$.
}
\label{f:diagram}
\end{figure}

\begin{lemma}\label{l:equiv}
The criterion \cref{e:Qincr} holds if and only if the functional $\bQ(\vphi)$ satisfies
\begin{align}\label{e:barQ}
\frac{\dif}{\dif h}\bar\Q(\varphi)<0,\quad \bar\Q(\varphi):=2\int_{\R}\frac{(-\varphi-\ln(1-\varphi))^{1/2}}{1-\varphi}\,\dif\tau,
\end{align}
where $\vphi=\vphi(\bpsi,h)$ solves \cref{e:sepHbar}, i.e., $(\vphi,\bpsi)\in\Gammah$
and $\vphi\in(0,1)$, and $h\in(0,2)$ defined as in \cref{e:sepHbar}.
\end{lemma}

\begin{remark}\label{r:transform}
The significance of the transformations from the original system \cref{e:ODE1} to the Hamiltonian system \cref{e:ex3} is purely technical, i.e., this can simplify the computation of $\frac{\dif}{\dif c}\Q(\phi,c)$ in \cref{e:Qincr} and make it more likely to determine its positivity.
Specifically, in contrast to the integral $Q(\phi,c)$ in \cref{e:Qincr} that explicitly depends on the variable $c$, the integral $\bar\Q(\varphi)$ in \cref{e:barQ} depends on the variable $h$ implicitly under the transformations. This benefits from the special form of the transformed Hamiltonian \cref{e:sepHbar}, which has the separate form in terms of the variables and the singularity at $\vphi=0$.
\end{remark}

\begin{proof}[Proof of \cref{l:equiv}]
Using the rescalings \cref{e:rescaledSW,e:trans}, and $\dif\tau=\frac{1}{2}\sqrt{-\varphi-\ln(1-\varphi)}\,\dif\xi$,
the functional $Q(\phi,c)$ defined as in \cref{e:Qincr} can be transformed into
\begin{align*}
\Q(\phi,c)\ &=\ \int_{\mathbb{R}}\left(\frac{c-\kappa}{c-\phi}\ln\left(\frac{c-\kappa}{c-\phi}\right)
-\frac{c-\kappa}{c-\phi}+1\right)\,\dif\xi\\
&=\  \int_{\mathbb{R}}\left(\frac{1}{1-\varphi}\ln\left(\frac{1}{1-\varphi}\right)-\frac{1}{1-\varphi}+1\right)\dif\xi\\
&=\ 2\int_{\R}\frac{(-\varphi-\ln(1-\varphi))^{1/2}}{1-\varphi}\,\dif\tau\ =\ \bar\Q(\varphi).
\end{align*}
Since $\gamma = c-\kappa$ from \cref{e:rescaledSW} and $h = 2\kappa/\gamma$ on $\Gammah$, we have $h = 2\kappa/(c-\kappa)$ and it leads to $\dif h/\dif c = -2\kappa(c-\kappa)^{-2} < 0$. Hence the statement holds.
\end{proof}

\begin{lemma}\label{l:criterion}
The criterion \cref{e:barQ} holds for $\vphi\in(0,1)$ and $h$ as in \cref{l:equiv}.
\end{lemma}

\begin{proof}
We rewrite $\bQ(\vphi)$ as the following
\begin{align*}
\bar{Q}(\varphi)&=2\int_{\mathbb{R}}\frac{(-\varphi-\ln(1-\varphi))^{1/2}}{1-\varphi}\,\dif\tau\\
&=-2\int_{\Gammah}\frac{(-\varphi-\ln(1-\varphi))^{5/2}}
{\varphi f(\varphi)}\,\dif\bpsi = -2\int_{\Gammah}g(\varphi)\,\dif\bpsi =:\cQ(h),
\end{align*}
where $g(\vphi) := (-\varphi-\ln(1-\varphi))^{5/2}/(\varphi f(\varphi))$ and the second equality stems from the change of variables in the second equation of \cref{e:ex3}.

Multiplying both sides of the last equality presented above by $h$ and using the expression of $h=H(\vphi,\bpsi)$ from \cref{e:sepHbar}, where $H(\varphi,\bpsi)$ is constant along the trajectory $\Gammah$, as well as using the expression of $\dif\vphi/\dif\bpsi$ given by the quotient of two equations in \cref{e:ex3}, yields
\begin{align}
h\cQ(h)&=-2h\int_{\Gammah}g(\varphi)\,\dif\bpsi
=-2\left(\int_{\Gammah}-\frac{\varphi^2g(\varphi)}{\varphi+\ln(1-\varphi)}\,\dif\bpsi
+\int_{\Gammah}-\bpsi^2 g(\varphi)\,\dif\bpsi\right) \nonumber\\
&=-2\left(\int_{\Gammah}\frac{\varphi(-\varphi-\ln(1-\varphi))^{3/2}}
{f(\varphi)}\,\dif\bpsi
+\int_{\Gammah}\frac{(-\varphi-\ln(1-\varphi))^{1/2}\bpsi}{2(1-\varphi)}\,\dif\varphi\right) \nonumber\\
&\triangleq -2\big(I_1(h)+I_2(h)\big), \label{e:hbarQ}
\end{align}
Differentiating both sides of \cref{e:hbarQ} with respect to $h$, with the notation $':=\dif/\dif h$, yields
\begin{equation*}
\cQ(h)+ h\cQ'(h)=-2\big(I'_1(h)+I'_2(h)\big),
\end{equation*}
where
\begin{equation*}
\begin{split}
I'_1(h)=&\ \int_{\Gammah}\left(\frac{\partial}{\partial \vphi}\left(\frac{\varphi\big(-\varphi-\ln(1-\varphi)\big)^{3/2}}
{f(\varphi)}\right)\cdot\frac{\partial\varphi}{\partial h}\right)\dif\bpsi\\
=&\ \int_{\Gammah}\Bigg(-\frac{(-\varphi-\ln(1-\varphi))^{5/2}}
{2\varphi f^3(\varphi)}\cdot\Big(4\varphi\ln^2(1-\varphi)-4\ln^2(1-\varphi)\\
&\ \qquad-4\varphi^3\ln(1-\varphi)+12\varphi^2\ln(1-\varphi)-8\varphi\ln(1-\varphi)-\varphi^4
+8\varphi^3-4\varphi^2\Big)\Bigg)\,\dif\bpsi,
\end{split}
\end{equation*}
and
\begin{equation*}
\begin{split}
I'_2(h)&=\int_{\Gammah}\left(\frac{(-\varphi-\ln(1-\varphi))^{1/2}}{2(1-\varphi)}
\cdot\frac{\partial\bpsi}{\partial h}\right)\dif\varphi=\int_{\Gammah}\left(\frac{(-\varphi-\ln(1-\varphi))^{1/2}}{2(1-\varphi)}\cdot\left(-\frac{1}{2\bpsi}\right)\right)\dif\varphi\\
&=\int_{\Gammah}\frac{(-\varphi-\ln(1-\varphi))^{5/2}}
{2\varphi f(\varphi)}\,\dif\bpsi,\\
\end{split}
\end{equation*}
where $\partial\vphi/\partial h$ and $\partial\bpsi/\partial h$ can be solved by differentiating \cref{e:sepHbar} with respect to $h$. In addition, we note that the integrals $I_j(h),I'_j(h),j=1,2$ are Riemann integrals since their integrands are bounded for $\vphi\in(0,\vphi_0]$, i.e., for $(\vphi,\bpsi)\in\Gammah$.
It follows that
\begin{align*}
h\cQ'(h)=&\ -2(I'_1(h)+I'_2(h))-\cQ(h)
= \int_{\Gammah}G(\vphi) F(\vphi)\,\dif\bpsi\\
=&\ \int_a^{-a} G(\vphi) F(\vphi)\,\dif\bpsi
= -2\int_0^{a}G(\vphi)F(\vphi)\,\dif\bpsi.
\end{align*}
where $a=\sqrt{2-h}$, and $G(\vphi):= (-\vphi - \ln(1-\vphi))^{5/2}/f^3(\vphi)$ with $G(\vphi)>0$ for $\vphi\in(0,1)$, and
\begin{equation*}
F(\varphi):=4\left(\varphi \ln^2(1-\varphi)-\ln^2(1-\varphi)+\varphi^2\right).
\end{equation*}
Since $h>0$, a sufficient condition for $\cQ'(h)<0$ is that $F(\vphi)>0$ for $\vphi\in(0,1)$, and in the remainder of this proof we will show the latter.

Differentiating $F(\varphi)$ twice with respect to $\varphi$, with abuse of notation $':=\dif/\dif\vphi$, yields
\begin{align*}
F'(\vphi) = 8\vphi + 4\ln(1-\vphi) (2 + \ln(1-\vphi)),\quad F''(\varphi) =\frac{8(\varphi+\ln(1-\varphi))}{\varphi-1}.
\end{align*}
We consider only $\vphi\in(0,1)$. We note that $F''(\vphi)>0$
since $\vphi+\ln(1-\vphi)$ is negative as mentioned above, and it follows that $F'(\varphi)$ is monotonically increasing in $\vphi$. Since $\lim_{\varphi\to0}F'(\varphi)=0$, we have $F'(\varphi)>0$,
which means that $F(\varphi)$ monotonically increases in $\vphi$. Combining this with the fact that $\lim_{\varphi\to0}F(\varphi)=0$, we obtain $F(\varphi)>0$.
\end{proof}

\section{Discussion}\label{s:discussion}

We have studied the orbital stability of smooth solitary waves of the $b$-family of Camassa-Holm equations for $0<b\leq1$. For $0<b<1$, the analysis of the stability is analogous to that of the case $b>1$ presented in \cite{LP2022}, and the stability criterion has been verified analytically by \cite{LL2023}. By comparison, the Hamiltonian structure is different for the case of $b=1$, and thus we have provided a rigorous proof of the orbital stability of solitary waves of \cref{e:CH}, which is based on a general framework motivated by the approach of \cite{GSS1987}. We have further verified the stability criterion \cref{e:Qincr} using the method analogous to that in \cite{LL2023}, which is again the application of the framework of monotonicity of the period function in planar Hamiltonian systems.

We also expect that the combination of the approach of \cite{GSS1987} and the method of monotonicity of the period function may be applied to the stability analysis of \cref{e:bCH} for $b<0$, and it would be interesting to pursue this further.
Moreover, we expect this method to be effective not only for
the $b$-CH equations, but also for other types of Hamiltonian systems.
Another interesting direction would be considering the orbitally asymptotic stability of solitary waves of \cref{e:bCH} with respect to small perturbations in some function spaces, cf., e.g., \cite{PW1994} for KdV equations and \cite{Molinet18,Molinet20} for peakons of CH equations.

\appendix

\section{Monotonicity of $G_{c,\kappa}$}\label{s:monotone}

For convenience, we suppress the subscript of $G_{c,\kappa}$. Differentiating the both sides of $E(G,0) = E_\hom$ with respect to $c$, and solving $\partial_c G$, yields
\begin{align*}
\partial_c G = \frac{\kappa(c-G)}{(G-\kappa)(G-c+\kappa)}\left(\frac{c-\kappa}{c-G} - \ln\left(\frac{c-\kappa}{c-G}\right) - 1\right)>0,
\end{align*}
since $\kappa<c-\kappa<G<c$ and the function $f(w) := w - \ln w - 1 >0$  for $w>1$. Analogously, differentiating both sides of $E(G,0) = E_\hom$ with respect to $\kappa$ and solving $\partial_\kappa G$, yields
\begin{align*}
\partial_\kappa G = -\frac{(c-2\kappa)(c-G)}{(G-c+\kappa)(G-\kappa)}\ln\left(\frac{c-\kappa}{c-G}\right) < 0.
\end{align*}

\section{Operator $\Jm$}\label{s:opJ}
We recall that $m\in\calX_\kappa$ defined in \cref{e:mspace}, and denote $\tm:= m-\kappa\in\Hspace^1(\RR)$. Let $\psi\in\Hspace^1(\RR)$, we have $\partial_x\psi\in\Lspace^2(\RR)$ and denote $v:=\partial_x\psi$. Then the product $mv\in\Lspace^2(\RR)$ since
\begin{align*}
\|mv\|_2^2 = \int_\RR |\tm v+\kappa v|^2\,\dif x &\leq 2 \int_\RR \left(|\tm v|^2 + |\kappa v|^2\right)\dif x\\
& \leq 2(\|\tm\|_\infty^2 + \kappa^2)\|v\|_2^2 \leq 2(\|\tm\|_{\Hspace^1}^2 + \kappa^2)\|v\|_2^2,
\end{align*}
which follows from the Sobolev embedding $\Hspace^1(\RR)\subset\Lspace^\infty(\RR)$. Since the operator $\partial_x:\Hspace^1(\RR)\to\Lspace^2(\RR)$ has no kernel, it is invertible and the range of $\partial_x^{-1}$ lies in $\Hspace^1(\RR)$, which leads to $\partial_x^{-1}(mv)\in\Hspace^1(\RR)$. We denote $w:=\partial_x^{-1}(mv)$ and recall the representation of $(1-\partial_x^2)^{-1}$ in \cref{e:convolution}, then we have $(1-\partial_x^2)^{-1}w\in\Hspace^1(\RR)$ since
\begin{align*}
\|(1-\partial_x^2)^{-1}w\|_{\Hspace^1}^2 &= \|g*w\|_{\Hspace^1}^2 = \|g*w\|_2^2 + \|\partial_x(g*w)\|_2^2 = \|g*w\|_2^2 + \|g*\partial_x w\|_2^2\\
&\leq \|g\|_1^2\|w\|_2^2 + \|g\|_1^2\|\partial_xw\|_2^2 = \|g\|_1^2\|w\|_{\Hspace^1}^2,
\end{align*}
where $g\in\Lspace^1(\RR)$.
Denote $z:=(1-\partial_x^2)^{-1}w$, since $\Hspace^1(\RR)$ is a Banach algebra, we have
\begin{align*}
\|m z\|_{\Hspace^1}^2 &= \|m z\|_2^2 + \|\partial_x(m z)\|_2^2 = \int_\RR |\tm z+\kappa z|^2\,\dif x + \int_\RR |\partial_x(\tm z)+\kappa \partial_x z|^2\,\dif x \\
&\leq 2\int_\RR \left(|\tm z|^2+|\kappa z|^2\right)\dif x + 2\int_\RR \left(|\partial_x(\tm z)|^2 + |\kappa \partial_x z|^2\right)\dif x  = 2\|\tm z\|_{\Hspace^1}^2 + 2\kappa^2\|z\|_{\Hspace^1}^2\\
&\leq 2\|\tm\|_{\Hspace^1}^2\|z\|_{\Hspace^1}^2 + 2\kappa^2\|z\|_{\Hspace^1}^2,
\end{align*}
i.e., $mz\in\Hspace^1(\RR)$.
It follows that $\partial_x(mz)\in\Lspace^2(\RR)$, which means $\Jm\psi\in\Lspace^2(\RR)$.

\section{Invertibility of $\Jm$}\label{s:nokernel}
We rewrite the operator $\Jm$ in \cref{e:J} as
\begin{align*}
\Jm = -\partial_x\widetilde\Jm\ &:\ \Hspace^1(\RR) \to \Lspace^2(\RR),\\
\widetilde\Jm := m(1-\partial_x^2)^{-1}\partial_x^{-1}(m\partial_x\cdot)\ &:\  \Hspace^1(\RR) \to \Hspace^1(\RR),
\end{align*}
with $m\in\calX_\kappa$ defined in \cref{e:mspace}. We assume that $\Jm$ has a kernel $\psi\in\Hspace^1(\RR)$. Since the operator $\partial_x$ has no kernel on $\Lspace^2(\RR)$, $\psi\in\ker(\widetilde\Jm)$.
Since
$m(x)>0$
for all $x\in\RR$ and the operator $1-\partial_x^2$ has no kernel on $\Lspace^2(\RR)$, we deduce that
\begin{align*}
(1-\partial_x^2)^{-1}\partial_x^{-1}(m\partial_x \psi) = 0\quad \Rightarrow \quad m\partial_x\psi = 0\quad \Rightarrow\quad \partial_x\psi = 0,
\end{align*}
which contradicts the fact that
$\partial_x$ has no kernel on $\Lspace^2(\RR)$.

\section{Variational derivative of $\calQ$}\label{s:Casimirs}
For convenience, we denote $q:=\frac{\delta \calQ}{\delta m}(m)$. We may solve $q$ from the following equation
\begin{align*}
\Jm q = -\partial_x\big(m(1-\partial_x^2)^{-1}\partial_x^{-1}(m\partial_x q)\big) = \partial_x m.
\end{align*}
Integrating both sides with respect to $x$, and using the notation $\widetilde\Jm$ in \cref{s:nokernel}, yields
\begin{align*}
\widetilde\Jm q = m(1-\partial_x^2)^{-1}\partial_x^{-1}(m\partial_x q) = -m + k_2,
\end{align*}
where the integration constant $k_2=\kappa$ since the range of $\widetilde\Jm$ lies in $\Lspace^2(\RR)$ and $m\in\calX_\kappa$.
Since $m(x)>0$, we can apply the inverse operator $m^{-1}\partial_x(1-\partial_x^2)m^{-1}$ to the both sides, yields
\begin{align*}
\partial_x q \ =\ k_2(-m^{-3}m_x + 6m^{-5}m_x^3 - 6m^{-4}m_x m_{xx} + m^{-3}m_{xxx}).
\end{align*}
Then integrating both sides with respect to $x$, yields
\begin{align*}
q \ &=\ -\frac 1 2 k_2(-m^{-2} + 3m^{-4}m_x^2 - 2m^{-3}m_{xx}) + k_1\\
&=\ -\frac 1 2 k_2 \frac{\delta \calQ_2}{\delta m}(m) + k_1\frac{\delta \calQ_1}{\delta m}(m),
\end{align*}
where $k_1 = -1/(2\kappa)$ since $q\in\Hspace^1(\RR)$,
and
\[
\frac{\delta \calQ_2}{\delta m}(m) := -m^{-2} + 3m^{-4}m_x^2 - 2m^{-3}m_{xx},\quad \frac{\delta \calQ_1}{\delta m}(m) := 1,
\]
are the variational derivatives of $\calQ_2$ and $\calQ_1$ defined in \cref{e:Qj}, and the selections of $k_1,k_2$ above give rise to the form of $\calQ$ in \cref{e:charge}.

\section{Proof of estimate \cref{e:remainder}}\label{s:estimate}
The remainder term $\calR$ in \cref{e:expandLambda} has the expression
\begin{align*}
\calR &= \Lambda(\mu+h) - \Lambda(\mu) - \left\langle \frac{\delta\LG}{\delta m}(\mu), h \right\rangle - \frac 1 2 \left\langle \frac{\delta^2\LG}{\delta m^2}(\mu)h, h \right\rangle\\
&= \int_\RR \Big(h (f_{30}h^2 + 2f_{21} h h_\xi + f_{12} h_\xi^2) + \calO(h^2(h^2+h_\xi^2))\Big) \dif\xi,
\end{align*}
where the prefactors $f_{30}:= - \mu^{-2}/6 - \mu^{-4} - 10\mu^{-6}\mu_\xi^2$, $f_{21}:= 6\mu^{-5}\mu_\xi$, $f_{12}:= -3\mu^{-4}$, which are bounded for all $\xi\in\RR$ (but not uniformly in parameters $c$ or $\kappa$) due to the expressions in \cref{e:murelation} with bounded $\phi,\phi_\xi$. Using Cauchy's inequality and the Sobolev embedding $\Hspace^1(\RR)\subset\Lspace^\infty(\RR)$,
the remainder term $\calR$ satisfies the estimate
\begin{align*}
|\calR|
&\leq \int_\RR \Big(|h| \big(|f_{30}| + |f_{21}|)h^2 + (|f_{12}|+|f_{21}|) h_\xi^2\big) + \calO(h^2(h^2+h_\xi^2))\Big) \dif\xi\\
&\leq C_1\|h\|_\infty\|h\|_{\Hspace^1}^2 + \calO(\|h\|_\infty^2\|h\|_{\Hspace^1}^2) \leq C_1\|h\|_{\Hspace^1}^3 + \calO(\|h\|_{\Hspace^1}^4)
\end{align*}
for the constant $C_1:=\sup_{\xi\in\RR}\{|f_{30}|+|f_{21}|,|f_{12}|+|f_{21}|\}$.

\section{Proof of \cref{l:coercivity}}\label{s:coercivity}

To obtain the coercivity of $\calL$, we have to deal with its kernel and the eigenfunction associated with the negative eigenvalue. For convenience, we use the notations of function spaces as in \cref{e:fctspace} in the proof.

We first consider the kernel of $\calL$. We introduce the spectral projection onto the kernel of $\calL$, i.e., $P_0:\Xspace\to \ker(\calL)$
\begin{align*}
P_0 u = \frac{\langle u,\mu_\xi \rangle}{\langle \mu_\xi,\mu_\xi \rangle}\mu_\xi,
\end{align*}
and its complement $\Pi_0 = \id - P_0$ which has the kernel spanned by $\mu_\xi$,
and has the range
\begin{align*}
\Xspace_0 = \{h\in\Xspace: \langle h, \mu_\xi \rangle =0\} = \ker(\calL)^\perp,
\end{align*}
i.e., $\Pi_0:\Xspace\to\ker(\calL)^\perp$ since $\langle \Pi_0 u,\mu_\xi \rangle = 0$ for any $u\in\Xspace$. We observe that
\begin{align*}
\calA_c' = {\rm span}\{\mu_\xi\}^\perp \cap {\rm span}\{\psi_\calQ\}^\perp \cap \Yspace,
\end{align*}
and decompose the new space $\Yspace_0:=\Xspace_0\cap \Yspace$ as the following
\begin{align}\label{e:decomposeY}
\Yspace_0 = \calA_c' \oplus {\rm span}\{\psi_\calQ\}.
\end{align}
Since both $P_0$ and $\Pi_0$ are self-adjoint operators and $\Pi_0$ is an identity operator on $\Yspace_0$, for any $u,v\in\Yspace_0$ we have the relations
\begin{align*}
\langle \calL u,v \rangle = \langle \calL u, \Pi_0 v \rangle = \langle \Pi_0\calL u, v \rangle.
\end{align*}
This motivates the definition $\calL_0 :=\Pi_0\calL$, which maps $\calL_0:D(\calL)\cap\Yspace_0\subset \Xspace_0\to \Xspace_0$. The spectral projection $P_0$ eliminates the kernel of $\calL$ from the space $\Xspace$, which leads to the spectrum $\sigma(\calL_0) = \sigma(\calL)\setminus\{0\}$, and thus $\calL_0$ is boundedly invertible which has a single negative eigenvalue $\lambda_0$ associated with eigenfunction $\psi_0$, with the rest of $\sigma(\calL_0)$ strictly positive, cf.,~\cref{l:opL}.

We next consider the bilinear form $\langle \calL u,v \rangle$ that is constrained on the space $\calA_c'$. We introduce the self-adjoint projection
\begin{align}\label{e:projPi}
\Pi u = u - \frac{\langle u,\psi_\calQ \rangle}{\langle \psi_\calQ, \psi_\calQ \rangle}\psi_\calQ.
\end{align}
Since $\psi_\calQ$ is not the eigenfunction of $\calL$, $\Pi$ is not a spectral projection for $\calL$, nor for $\calL_0$, and thus the projection $\Pi$ does not remove any eigenvalue from $\calL$, nor from $\calL_0$.
The parities of $\mu$ and its derivatives claimed in \cref{r:parity} implies that
\begin{align}\label{e:orthogonal}
\langle \psi_\calQ, \mu_\xi \rangle=0,
\end{align}
since $\psi_\calQ\mu_\xi$ is odd with mean zero.
Hence, we deduce that
$\Pi_0\Pi = \Pi\Pi_0$, and thus $\Pi:\Yspace_0\to \calA_c'$; moreover, $\Pi\Pi_0$ is the identity operator on $\calA_c'$.
Hence, for any $u,v\in\calA_c'$, the relations
\[
\langle \calL u,v \rangle = \langle \calL u, \Pi\Pi_0 v \rangle = \langle \Pi\calL u, \Pi_0 v \rangle = \langle \Pi_0\Pi\calL u, v \rangle = \langle \Pi\Pi_0\calL u, v \rangle
\]
motivates the definition $\calL_\Pi := \Pi\calL_0$, which maps
$\calL_\Pi: D(\calL)\cap\calA_c'\subset\Xspace_0 \to \Pi\Xspace_0$.
The constrained operator $\calL_\Pi$ is self-adjoint, so its spectrum is real. The difference $\calL_0-\calL_\Pi$, acting on $\calA_c'$ is rank-one since its range is spanned by $\psi_\calQ$, and hence compact, so the Weyl essential spectrum theorem implies that $\sigma_\ess(\calL_\Pi) = \sigma_\ess(\calL_0)$.

Concerning the point spectrum of $\calL_\Pi$, the smallest eigenvalue $\alpha_0$ of $\calL_\Pi$ has the variational characterization
\begin{align}\label{e:character}
\alpha_0 = \inf_{\psi\in\calA_c'}\frac{\langle \calL\psi,\psi \rangle}{\langle \psi,\psi \rangle}.
\end{align}
The smallest eigenvalue $\lambda_0<0$ of $\calL$ has a similar formulation, with the infimum over $\psi\in\Yspace$. We deduce that $\lambda_0\leq\alpha_0$, since $\lambda_0$ is the infimum of the same functional over a larger space. Since $\psi_0$ has no zeros (cf.,~\cref{l:opL}) and $\psi_\calQ$ is of one sign over $\xi\in\RR$ (cf.,~\cref{r:sign}),
we have $\langle \psi_0,\psi_\calQ \rangle\neq 0$, which implies $\psi_0\notin\calA_c'$, and hence $\lambda_0<\alpha_0$. Indeed, a purpose of the constraint $\langle \cdot,\psi_\calQ\rangle=0$ is to eliminate the eigenfunction of $\calL$ associated with negative eigenvalue from the function space $\Yspace_0$. The above infimum is achieved at $\psi\in\calA_c'$, which solves the eigenvalue problem
\begin{align*}
\calL_\Pi\psi = \alpha_0\psi\quad \Rightarrow\quad \calL_0\psi = \alpha_0\psi + \nu\psi_\calQ,
\end{align*}
for some value of the multipliers $\alpha_0,\nu\in\RR$.
Since $\calL_0$ is invertible on $\Xspace_0$, we may solve for $\psi\in\Xspace_0$
\begin{align}\label{e:psi}
\psi(\alpha_0) = \nu(\calL_0-\alpha_0)^{-1}\psi_\calQ.
\end{align}

To enforce the condition $\psi\in\calA_c'$, we impose the constraint $\langle \psi,\psi_\calQ \rangle = 0$.
We use \cref{e:psi} to rewrite the constraint as a function of $\alpha_0$,
\begin{align*}
g(\alpha_0):= \frac{1}{\nu}\langle \psi,\psi_\calQ \rangle = \langle (\calL_0-\alpha_0)^{-1}\psi_\calQ,\psi_\calQ \rangle,
\end{align*}
so that $g(\alpha_0)=0$ if and only if $\alpha_0$ is an eigenvalue of $\calL_\Pi$. In particular, if $g(\alpha_0)$ is nonzero for all $\alpha_0\leq 0$, then the smallest eigenvalue of $\calL_\Pi$ must be positive. The function $g$ is analytic for $\alpha_0\notin\sigma(\calL_0)$, with pole singularities possible at the eigenvalues of $\calL_0$, and is strictly increasing where it is smooth, since
$g'(\alpha_0) = \langle (\calL_0-\alpha_0)^{-2}\psi_\calQ, \psi_\calQ \rangle = \|(\calL_0-\alpha_0)^{-1}\psi_\calQ\|_2^2>0$ using the Green's function for the resolvent.
The operator $\calL_0$ has no spectrum on $(\lambda_0,\lambda_1)$, where $\lambda_1>0$ is the smallest element of $\sigma(\calL_0)\setminus\{\lambda_0\}$. Consequently, the function $g$ is analytic and monotonically increasing on this interval. It follows that the first zero, $\alpha_0$, of $g$ is positive if $g(0)<0$, and this zero is the smallest eigenvalue of $\calL_\Pi$.

To relate the value of $g(0)$ to the derivative of the charge $\calQ(\mu)$ with respect to wave speed, we differentiate \cref{e:steadyLambda}
with respect to $c$\begin{align*}
\partial_c\frac{\delta\LG}{\delta m}(\mu) = \partial_c\left(- \frac{\delta\calH}{\delta m}(\mu) - c\frac{\delta \calQ}{\delta m}(\mu)  \right) = \calL\partial_c\mu - \psi_\calQ =0 \quad \Rightarrow\quad \calL\partial_c\mu = \psi_\calQ.
\end{align*}
We recall that $\langle \psi_\calQ, \mu_\xi \rangle=0$ in \cref{e:orthogonal}, thus $\psi_\calQ\in\Xspace_0$ and equivalently $\Pi_0\psi_\calQ=\psi_\calQ$.
Applying $\Pi_0$ to both sides of the above equation and inverting $\calL_0$ on $\Xspace_0$, yields
\begin{align*}
\calL_0\partial_c\mu = \psi_\calQ\quad \Rightarrow\quad \partial_c\mu = \calL_0^{-1}\psi_\calQ,
\end{align*}
This allows us to rewrite $g(0)$ as
\begin{align*}
g(0) = \langle \calL_0^{-1}\psi_\calQ,\psi_\calQ \rangle = \langle \partial_c\mu,\psi_\calQ \rangle = \partial_c \calQ(\mu).
\end{align*}
Hence $\partial_c \calQ(\mu)<0$ implies that $\alpha_0>0$.

Up to now, in the case $\alpha_0>0$ the variational characterization \cref{e:character} gives the $\Lspace^2$-coercivity
\begin{align*}
\langle \calL v, v \rangle \geq \alpha_0 \|v\|_2^2,\quad v\in\calA_c'.
\end{align*}
To obtain the $\Hspace^1$-coercivity we argue by contradiction. Assume the bilinear form $\langle \calL v, v \rangle$
is not $\Hspace^1$-coercive over $\calA_c'$, then for any $\eps>0$, there is a sequence of functions $\{v_n\}_{n=1}^\infty\subset\calA_c'\subset\Hspace^1(\RR)$ with $\|v_n\|_{\Hspace^1} = 1$ such that $\langle \calL v_n,v_n \rangle<\eps\|v_n\|_{\Hspace^1}^2$, i.e., $\langle \calL v_n,v_n \rangle\to0$ as $n\to\infty$. The $\Lspace^2$-coercivity $\langle \calL v_n,v_n \rangle\geq \alpha_0\|v_n\|^2_2$ implies that $\lim_{n\to\infty}\|v_n\|_2 = 0$, and since $\|v_n\|_{\Hspace^1} = 1$ we have
\begin{align*}
\lim_{n\to\infty}\langle \partial_\xi v_n, \partial_\xi v_n \rangle = 1.
\end{align*}
Recalling the form of $\calL$ in \cref{e:L} and suppressing the prefactor $c\kappa$, we have the limit
\begin{align*}
\lim_{n\to\infty}\langle -\partial_\xi(\mu^{-3}\partial_\xi v_n), v_n \rangle = \lim_{n\to\infty}\langle \mu^{-3}\partial_\xi v_n, \partial_\xi v_n \rangle \geq M_{c,\kappa}^{-3}\lim_{n\to\infty}\langle \partial_\xi v_n, \partial_\xi v_n \rangle = M_{c,\kappa}^{-3},
\end{align*}
where $\mu(\xi)\leq M_{c,\kappa}$ for $\xi\in\RR$, see \cref{r:rangemu} for more details of $M_{c,\kappa}$.
For the potential $W(\xi):= \mu^{-3} - 6\mu^{-5}\mu_\xi^2 + 3\mu^{-4}\mu_{\xi\xi} - (c\kappa\mu)^{-1}$ associated with $\calL$ we have the limit
\begin{align*}
\lim_{n\to\infty}|\langle W(\xi)v_n,v_n \rangle| \leq \|W\|_\infty\lim_{n\to\infty}\langle v_n,v_n \rangle = 0.
\end{align*}
These limits bring us to the contradiction, namely
\begin{align*}
0 = \lim_{n\to\infty}\langle \calL v_n,v_n \rangle = \lim_{n\to\infty}\left( \langle -\partial_\xi(\mu^{-3}\partial_\xi v_n) , v_n \rangle + \langle W(\xi)v_n,v_n \rangle\right) \geq M_{c,\kappa}^{-3}>0.
\end{align*}
Hence, we obtain the $H^1$-coercivity \cref{e:H1coercivity} with some $\alpha>0$.

\subsection*{Acknowledgments}
C. Liu and T. Long are  supported by the National Natural Science Foundation of China (No. 12171491). T. Long is supported by China Postdoctoral Science Foundation (No. GZC20230900). The authors thank St\'ephane Lafortune for the inspiring discussions and comments. We thank the anonymous reviewers for comments that helped to improve the manuscript.


\end{document}